\documentclass[12pt,leqno]{article}

\input{style}  

\begin{document}

\title{Convex duality for stochastic singular
  control problems}
 
\author{P. Bank\\ Technische
  Universit{\"a}t Berlin\\
  Institut f{\"u}r Mathematik\\ Stra{\ss}e des 17. Juni 136, 10623
  Berlin, Germany\\
 \and
  H. Kauppila\\
  Columbia University in the City of New York\\
  Department of Mathematics\\
  2990 Broadway, New York, NY 10027
}
 \date{\today}

\maketitle
\begin{abstract}
  We develop a general theory of convex duality for certain singular
  control problems, taking the abstract results by~\citet{KramSch:99}
  for optimal expected utility from nonnegative random variables to
  the level of optimal expected utility from increasing, adapted
  controls. The main contributions are the formulation of a suitable
  duality framework, the identification of the problem's dual
  functional as well as the full duality for the primal and dual value
  functions and their optimizers. The scope of our results is
  illustrated by an irreversible investment problem and the
  Hindy-Huang-Kreps utility maximization problem for incomplete
  financial markets.
\end{abstract}
 
\begin{description}
\item[Keywords:] Convex duality, singular control, utility
  maximization, incomplete markets, irreversible investment.
\item[JEL Classification:] G11, G12, C61.
\item[AMS Subject Classification (2010):] 93E20, 91G80, 46N10, 91B08.
\end{description}

\section{Introduction}

A typical stochastic optimal control problem is formulated by
specifying how the dynamics of a given system can be influenced by a
controller to optimize some performance criterion. In classical
stochastic control the controller directly affects the coefficients
which govern the system's dynamics, but has no direct influence on the
system's state itself. In singular control problems, the controller
can, by contrast, directly change the state of the controlled system
at any time in a fully scalable way, from infinitesimal to large
jumps.

Ever since the seminal work on such singular problems by
\citet{BenesSheppWitsenhausen80} the most commonly used approach is to
consider Markovian systems and use dynamic programming to derive and
then solve the problem's Hamilton-Jacobi-Bellman equation which comes
in the form of a free-boundary value problem. Alternatively, one can
resort to versions of Pontryagin's maximum principle as first
discussed for stochastic singular control by
\citet{CadenillasHaussmann94}. In either case, the derived
mathematical concepts do not immediately solve the problem, but merely
help to describe some of the solution's properties. A key challenge is
then to work out this description as neatly as possible. Clearly, this
task is made easier when, as we shall assume, control can only be
exerted in one direction. Problems of this type include the monotone
follower of, e.g., \citet{KaratzasShreve84}, some irreversible
investment problems as discussed in
Section~\ref{sec:irreversibleInvestment}, but also optimal investment
and consumption problems with so-called Hindy-Huang-Kreps utilities
which we cover in Section~\ref{sec:HHK}.

All of these problems can be cast as maximization problems for
functionals $\UU$ of the form
\begin{equation*}
  \UU(C) = \E \int_0^\infty U_t(C_t) \,d\mu_t 
\end{equation*}
where $C$ is from the class $\cC$ of nonnegative, increasing,
left-continuous adapted controls, $U_t(C_t)$ describes the predictable
utility obtained at time $t \geq 0$ from the cumulative control $C_t$
and where the optional random measure $\mu$ describes the weights
assigned to utilities at different times. 

It is the purpose of this paper to develop a theory of convex duality
for singular control problems with target functionals of the above
type. Indeed, under natural assumptions on $U$ and $\mu$, our first
main Theorem~\ref{thm:1} establishes the Legendre-Fenchel duality of
the functional $\UU$. For this we introduce the class $\cD$ of
nonnegative, decreasing, right-continuous processes $D$ as dual
variables with the pairing
\begin{equation*}
  \E\abr{C,D}=\E\int_{[0,\infty)} D_t\,dC_t
\end{equation*}
and we show that the Legendre-Fenchel transform
\begin{equation*}
  \VV(D) = \sup_{C \in \cC} \cbr{\UU(C)-\E\abr{C,D}}
\end{equation*}
coincides with the functional
\begin{equation*}
  \VV(D)  = \inf_{\delta \in \bdot{\cD}(D)} 
  \E \int_0^\infty V_t(\delta_t) \,d\mu_t, 
\end{equation*}
where $\bdot{\cD}(D)$ is a certain class of optional processes
associated with $D$ and where $V_t$ denotes the classical
Legendre-Fenchel transform of $U_t$.  Moreover, we show that the
minimizer for $\VV(D)<\infty$ can be constructed in terms of a certain
envelope process of the form $\breve{D} = \int_.^\infty U'(C^D)\,d\mu$
with $C^D \in \cC$ which is characterized uniquely by
\begin{equation*}
  \condexp{\breve{D}_t} \leq \condexp{D_t} \text{ for all } t \geq 0,
\end{equation*}
with `$=$' holding true whenever $C^D$ is increasing. We thus obtain a
full characterization of the maximizers for
\begin{equation*}
  \UU(C) - \E\abr{C,D} = \E \int_0^\infty U_t(C_t) \,d\mu_t -
  \E\int_0^\infty D_t\,dC_t,
\end{equation*}
a general form, for instance, of irreversible investment problems as
described in Section~\ref{sec:irreversibleInvestment}.

For the treatment of constrained problems such as the
Hindy-Huang-Kreps optimal investment and consumption problem of
Section~\ref{sec:HHK} we formulate the abstract utility maximization
problem with value function
\begin{equation*}
  \mathbf{u}(x) = \sup_{C \in \cC(x)} \UU(C)
\end{equation*}
where, for $x>0$, controls are constrained to lie in $\cC(x) \subset
\cC$. This is assumed to be a convex class of feasible controls for
which a polar relation with sets $\cD(y) \subset \cD$, $y>0$, can be
established. This leads to the dual problems with value
\begin{equation*}
  \mathbf{v}(y) = \inf_{D \in \cD(y)} \VV(D)
\end{equation*}
for $y>0$. 

The celebrated papers by \citet{KramSch:99,KramSch:03} develop convex
duality for similarly abstract utility maximization problems where
utility is obtained at a single point in time, which in our setting
amounts to the choice of $\mu$ as a Dirac measure at some point
$T>0$. This leads to the obvious challenge to develop a similar convex
duality theory for our singular framework. This challenge is taken up
by our second main result, Theorem~\ref{thm:2}. While our proof of
this result follows to some extent the very useful blue-print laid out
by \citet{KramSch:99}, there are a number of novel obstacles to
overcome along the way. These are a consequence of our central
constraint of increasing adapted controls which in the setting of
Kramkov and Schachermayer corresponds to the considerably simpler
restriction to nonnegative $\cF_T$-measurable random variables. This
also distinguishes our work from \citet{KaratZitk:03} who consider
utility from consumption at nonnegative rates, i.e., without the
monotonicity constraint of our singular control set.

Specifically, a first key difference is in the structure of the
Legendre-Fenchel transform of the utility functional under
consideration: For Kramkov and Schachermayer's $C \mapsto \E U_T(C_T)$
the dual functional is simply $D\mapsto \E V_T(D_T)$ whereas the dual
$\VV$ of our functional $\UU$ involves an infimum. As a consequence,
the connection between the dual value $\VV(D)$ and the dual variable
$D$ is not as straight forward as in \cite{KramSch:99} but has to be
described by our envelope process $\breve{D}$. Also, the process $C^D$
which is conjugate to $D$ in the Legendre-Fenchel duality cannot be
directly written in terms of $D$, by contrast to~\cite{KramSch:99}
where one merely has to invert $U'_T(C^D_T)=D_T$. In addition, the
dual problem is not strictly convex anymore, a property which is
needed for some of the arguments in~\citet{KramSch:99}. As a remedy,
we introduce a subclass of $\cD(y)$ which is sufficiently large to
include the solutions to the dual problem, but small enough to ensure
strict convexity of $\VV$ on this subclass. This allows us to
establish the continuous dependence of certain solutions to the dual
problem on the Lagrange parameter $y$. The final challenge is then to
show that the corresponding candidate solutions for the primal problem
are indeed feasible for the larger class of all dual variables
$\cD(y)$.  Here, we have to resort to the general Legendre-Fenchel
duality between $\UU$ and $\VV$ developed in our first main
result. Finally, the notion of reasonable asymptotic elasticity
identified by~\cite{KramSch:99} as a key assumption for general
well-posedness of utility maximization problems has to be adapted to
account for the possibly very different utility functions $U_t$ at
different time points $t \geq 0$. In fact, in line with, e.g.,
\citet{BouchardPham04} and \citet{Zitk:05} we do allow for time- and
scenario-dependent utility functions and a stochastic clock which
allows us to include the finite time horizon case in the infinite time
horizon formulation in a simple manner; see the end of
Section~\ref{sec:HHK}.

The paper is organized as follows. In Section~\ref{sec:setup} we
introduce the class of controls $\cC$ and the space of dual variables
$\cD$ along with the assumptions and definition of our utility
functional $\UU$ and its dual $\VV$. Section~\ref{sec:results} is
devoted to the presentation of our main duality results,
Theorems~\ref{thm:1} and~\ref{thm:2}. Section~\ref{sec:illustrations}
illustrates these findings by a general irreversible investment
problem and by the optimal consumption problem of Hindy, Huang, and
Kreps. Section~\ref{sec:proofs} contains the proofs our main
theorems. Appendix~\ref{sec:convex-envelopes} gives the construction
of our envelope process $\breve{D}$. Appendix~\ref{sec:minimax}
discusses \citet{Zitkovic:10}'s notion of convex compactness in the
new context of our class of controls $\cC$ and provides a minimax
theorem compatible with this generalized notion of compactness.




\section{Controls and their performance measure}
\label{sec:setup}

We start by describing the control set $\cC$ and its dual $\cD$ as
well as our target utility functional $\UU$ along with a dual
functional $\VV$. As usual, we let $(\Omega,\cF,\FF,\P)$ denote
throughout a filtered probability space describing a controller's
beliefs $\P$ about future events $\cF$ along with his information flow
$\FF=(\cF_t)_{t \geq 0}$, a complete, right-continuous filtration
where $\cF_0$ is generated by the $\P$-null sets.

\subsection{Controls and their duals}
\label{sec:controls}

The set of conceivable controls will be given by the class $\cC$ of
predictable processes $C:\Omega \times [0,\infty] \to [0,\infty]$ with
non-decreasing, left-continuous paths starting from $C_0=0$. As usual
exercising control incurs costs which will be described by dual
variables. A convenient set of such dual variables will turn out to be
the class $\cD$ of all $\cF\otimes \cB([0,\infty])$-measurable
processes $D:\Omega \times [0,\infty] \to [0,\infty]$ with
non-increasing, right-continuous paths ending in $D_\infty=0$. Indeed,
for any $C \in \cC$ and $D \in \cD$ we can define
\begin{displaymath}
  \abr{C,D} \set \int_{[0,\infty)} D_t \,dC_t = - \int_{(0,\infty]} C_t
  \,dD_t,
\end{displaymath}
which yields the pairing
\begin{equation}
  \label{eq:1}
  \E\abr{C,D} = \E\int_{[0,\infty)} D_t \,dC_t =-\E\int_{(0,\infty]} C_t
  \,dD_t\in [0,\infty].
\end{equation}
Observe that the above identities are to be understood and hold with
the following conventions regarding the integration with respect to $C
\in \cC$ and $D \in \cD$:
\begin{itemize}
\item $dC$ and $dD$ do not charge the intervals $(\inf\cbr{t \geq 0
    \;:\; C_t=\infty},\infty]$ and $[0,\sup\cbr{t \geq 0\;:\;D_t=\infty})$,
  respectively;
\item the integration with respect to $dC$ is carried out taking into
  account a point mass of size $C_{0+} \set \lim_{t \downarrow 0} C_t$
  at $0$ and the integration with respect to $dD$ assumes a point mass
  $D_{\infty-} \set \lim_{t \uparrow \infty} D_t$ at $\infty$;
\item we let $0 \cdot \infty \set 0$ should an integrand be zero where
  the integrator puts an infinite point mass.
\end{itemize}
Finally, we note that both $\cC$ and $\cD$ can be endowed with the
metric which for two $\cF \otimes \cB([0,\infty])$-measurable
processes $A$, $B$ assigns the distance
\begin{equation}
  \label{eq:2}
  \dist(A,B) \set \E \int_0^\infty |h(A_t)-h(B_t)|\,d\mu_t
\end{equation}
where $h$ is any homeomorphism $[-\infty,\infty]\to[0,1]$. With
respect to this distance the pairing~\eqref{eq:1} is
lower-semicontinuous in each of its factors; see Lemma~\ref{lem:11}.

\subsection{Utilities and their conjugates}

The performance of controls will be measured by the utilities they
provide at each time, weighted with the controller's time preferences.

\begin{Assumption}\label{asp:1}
  The controller's time preferences are described by an optional random
  measure $\mu$ on $[0,\infty)$ without atoms, full support and finite
  expected total mass $\E \mu([0,\infty)) <\infty$.
 
 The controller's utility is specified by a mapping
 \begin{align*}
   U:\Omega \times [0,\infty) \times [0,\infty) &\to [0,\infty)\\
(\omega,t,c) & \mapsto U_t(\omega,c)
 \end{align*}
  with the following properties:
  \begin{enumerate}
  \item For any $(\omega,t) \in \Omega \times [0,\infty)$,
    $U_t(\omega,.)$ is continuous, strictly concave and strictly
    increasing from $U_t(\omega,0)=0$ to
    $U_t(\omega,\infty) \set \lim_{c \uparrow \infty} U_t(\omega,c)\in
    [0,\infty]$. Moreover, $U_t(\omega,.)$ is continuously
    differentiable and satisfies the Inada conditions
    \begin{equation*}
      U_t'(\omega,0) \set \lim_{c \downarrow 0} U_t'(\omega,c)=\infty
      \text{ and } U_t'(\omega,\infty) \set \lim_{c \uparrow \infty} U'_t(\omega,c)=0.
    \end{equation*}
  \item For any $c \geq 0$, $(\omega,t) \mapsto U_t(\omega,c)$ is predictable with $\E
    \int_0^\infty U_t(c)\,d\mu_t<\infty$. \item The asymptotic
    elasticity of $U$ is uniformly less than one in
    the sense that there is a constant
    $\gamma \in (0,1)$ and a predictable process $C^\gamma \geq 0$ with $\E
    \int_0^\infty U_t(C^\gamma_t) \,d\mu_t<\infty$ such that for
    any  $(\omega,t) \in \Omega \times [0,\infty)$ we have 
    \begin{equation} \label{eq:3}
      \frac{cU_t'(\omega,c)}{U_t(\omega,c)} < \gamma<1 \text{ for all } c
      > C^\gamma_t(\omega).
    \end{equation}
  \end{enumerate}
\end{Assumption}

A control will provide an expected utility of the form
\begin{equation}
  \label{eq:4}
  \UU(C) \set \E \int_0^\infty U_t(C_t) \,d\mu_t \in [0,\infty], \quad C \in \cC.
\end{equation}
Note that in our setting, contrary to what is more commonly assumed,
utility $U_t(C_t)$ at each time $t \geq 0$ is obtained from the
\emph{cumulative} control $C_t$ rather than the current control
rate. This turns the optimization problem to be introduced shortly
into a singular stochastic control problem. We refer to the
illustrations of Section~\ref{sec:illustrations} for the motivation
and scope of such utility functionals.

\begin{Remark} 
  Let us briefly comment on our preference Assumption~\ref{asp:1}:
  \begin{enumerate}
  \item The first item just requires that time- and scenario-wise the
    utility function is standard, except for the requirement that
    utility at zero vanishes. In fact, this comes without loss of
    generality if $\E\int_0^\infty |U_t(0)| \,d\mu_t<\infty$ since
    then we can pass to $\tilde{U} \set U-U(0)$ without changing the
    utility maximization problem for~\eqref{eq:4}. 

  \item The predictability requirement in the second item is
    essentially without loss of generality since we could work with
    the predictable projection of any non-predictable field $(U(c),
    c\in[0,\infty))$ without changing~\eqref{eq:4}. This holds because
    controls are predictable and because time preferences are optional
    random measures without atoms.
  \item It is well-known from the work of
    \citet{KramSch:99} that asymptotic elasticity less
    than one is necessary to avoid ill-posed utility maximization
    problems. Their Lemma~6.3 shows that for any $(\omega,t) \in
    \Omega \times [0,\infty)$, condition~\eqref{eq:3} is equivalent to
    \begin{equation} \label{eq:5} 
       U_t(\omega,\lambda c) <
      \lambda^\gamma U_t(\omega,c) \text{ for all } \lambda > 1 \text{
        and all } c > C^\gamma_t(\omega).
    \end{equation}
  \item Our results will also allow us to treat the case of a possibly
    finite time horizon given by some stopping time $\tau$; see
    Remark~\ref{rem:1} below.
  \end{enumerate}
\end{Remark}

One of the main results in this paper is the convex duality for the
functional $\UU$ of~\eqref{eq:4} which will be established in
Theorem~\ref{thm:1} below. For this we need to introduce a dual
functional $\VV$ on $\cD$. This functional will be specified in terms
of the classical Legendre-Fenchel transform $V$ of $U$:
\begin{equation}
  \label{eq:6}
  V_t(\omega,d) \set \sup_{0 \leq c <\infty}\cbr{U_t(\omega,c)-c d}, \quad d>0.
\end{equation}
It is well-known that under the conditions in Assumption~\ref{asp:1}
$V_t(\omega,.)$ is a strictly convex and decreasing function on
$(0,\infty)$ with 
\begin{displaymath}
V_t(\omega,0) \set \lim_{d \downarrow 0}
V_t(\omega,d) = U_t(\omega,\infty) \text{ and } V_t(\omega,\infty) \set
\lim_{d \uparrow \infty} V_t(\omega,d) = U_t(\omega,0)=0. 
\end{displaymath}
Moreover, $V_t(\omega,.)$ is continuously differentiable on
$(0,\infty)$ and satisfies the Inada conditions
\begin{equation*}
  V_t'(\omega,0) \set \lim_{d \downarrow 0} V_t'(\omega,d)=-\infty
  \text{ and } V_t'(\omega,\infty) \set \lim_{d \uparrow \infty} V_t'(\omega,d)=0.
\end{equation*} 
The asymptotic elasticity conditions~\eqref{eq:3} and~\eqref{eq:5} can
be cast in terms of $V$ as
\begin{equation} \label{eq:7} 
  (1-\gamma)(-V'(d))d \leq \gamma V(d) 
 \text{  for all } 0 < d < D^{\gamma}
\end{equation} 
and
\begin{equation}\label{eq:8} 
  V((1-\epsilon) d) <
  (1-\epsilon)^{\frac{\gamma}{1-\gamma}} V(d) \text{ for all }
  0<\epsilon<1, \; 0<d < D^{\gamma}
\end{equation}
with the same $\gamma \in (0,1)$ as before and $D^{\gamma} \set
U'(C^{\gamma})$ ; see Lemma~6.3 in~\citet{KramSch:99}. Finally, along
with $U$ also $V$ is predictable and we have the following conjugacy
relations:
\begin{enumerate}
\item In addition to~\eqref{eq:6}, we also have
    \begin{equation}\label{eq:9}
    U_t(\omega,c) = \inf_{0 \leq d <\infty} \cbr{V_t(\omega,d)+cd}, \quad c>0.
      \end{equation}
\item The supremum in~\eqref{eq:6} is attained at $c=-V_t'(\omega,d)$.
\item The infimum in~\eqref{eq:9} is attained at $d=U_t'(\omega,c)$.
\end{enumerate}
In fact, the identities in items 2. and 3. are equivalent. 

We now can introduce the dual functional
\begin{equation}
  \label{eq:10}
  \VV(D) \set \inf_{\delta \in \bdot{\cD}(D)} 
  \E \int_0^\infty V_t(\delta_t) \,d\mu_t \in [0,\infty], \quad D \in \cD\,,
\end{equation}
with 
\begin{equation}
   \label{eq:11}
   \bdot{\cD}(D) \set \cbr{\delta\geq 0 \text{ optional} \;:\; 
     \opt{\left(\int_.^\infty \delta \,d\mu\right)} \leq \opt{D}},
 \end{equation}
 where we used the notation $\opt{X}$ for the optional projection of any
 $\cF \otimes \cB([0,\infty])$-measurable process $X\geq 0$.

\section{Main results}
\label{sec:results}

\subsection{Legendre-Fenchel duality for utility functionals}

For the statement of our duality theorem for $\UU$ and $\VV$
of~\eqref{eq:4} and~\eqref{eq:10} we have
to introduce for any dual process $D \in \cD$ a special envelope
process $\breve{D}$ of the form
\begin{equation}
  \label{eq:12}
  \breve{D}_t = \int_t^\infty U'(C^{\breve{D}})\,d\mu, \; t \geq 0, \text{ for some }
  C^{\breve{D}} \in \cC
 \end{equation}
which satisfies $\P$-almost surely
\begin{equation}
  \label{eq:13}
  \opt{\breve{D}}_t \leq \opt{D}_t \text{ for any $t \geq 0$, with ``$=$" if } dC^{\breve{D}}_t>0.
\end{equation}
Here, we follow the convention that, for $C \in \cC$, we write
$dC_t>0$ iff $t$ is a point of increase for $C$ in the sense that
$C_t<C_{s+t}$ for any $s > 0$. We refer to Lemma~\ref{lem:10} of
Appendix~\ref{sec:convex-envelopes} for existence and uniqueness up to
indistinguishability of such an envelope process.

Note that the paths of such an envelope process $\breve{D}$ are
absolutely continuous with respect to $\mu$. We choose
\begin{equation}
  \label{eq:14}
  \bdot{\breve{D}} \set  -U'(C^{\breve{D}})  
\end{equation} 
for the corresponding density which is then uniquely determined up to
indistinguishability because so is the process $C^{\breve{D}} \in \cC$
with~\eqref{eq:12} and~\eqref{eq:13}. Observe, that, conversely, we
can then write $C^{\breve{D}} = -V'(-\bdot{\breve{D}})$ by the
conjugacy relations between $U$ and $V$ recalled above.

We now can state our first main result as follows:

\begin{Theorem}\label{thm:1}
Under Assumption~\ref{asp:1} the following assertions hold:
\begin{enumerate}
\item The functionals $\UU$ of~\eqref{eq:4} and $\VV$
  of~\eqref{eq:10} are conjugate to each other in the sense that we
  have
  \begin{equation}
  \label{eq:15}
  \UU(\hat{C}) = \inf_{\VV(D)<\infty}
  \cbr{\VV(D)+\E\abr{\hat{C},D}} \text{ for any } \hat{C} \in \cC
 \end{equation}
 and
 \begin{equation}
   \label{eq:16}
   \VV(\hat{D})=\sup_{\UU(C)<\infty}
   \cbr{\UU(C)-\E\abr{C,\hat{D}}} \text{ for any } \hat{D} \in \cD\,.
 \end{equation}

\item If finite, the infimum in~\eqref{eq:15} is attained for
  precisely those $D \in \cD$ whose (joint) envelope process
  $\breve{D}$ with~\eqref{eq:12} and~\eqref{eq:13} is given by
  \begin{equation}
    \label{eq:17}
    \bdot{\breve{D}}= -U'(\hat{C}) .
  \end{equation}

\item If finite, the supremum in~\eqref{eq:16} is attained exactly for
  \begin{equation}
    \label{eq:18}
    \hat{C} =-V'(-\bdot{\breve{D}}) \in \cC
  \end{equation}
  where $\breve{D}$ is the envelope process of $\hat{D}$ characterized
  by~\eqref{eq:12} and~\eqref{eq:13} with $D\set\hat{D}$.
\end{enumerate}
\end{Theorem}

\subsection{Convex duality for an abstract utility maximization problem}
\label{sec:abstractDuality}

Let us now formulate an abstract utility maximization problem in a
similar way as in the approach for utility from terminal wealth by
\citet{KramSch:99}. To this end we consider $\cC(1) \subset \cC$ and
$\cD(1) \subset \cD$ which are polar with respect to each other in the
sense that
\begin{enumerate}
\item For any $C \in \cC$, we have $C \in \cC(1)$ iff $\E \abr{C,D}
  \leq 1$ for any $D \in \cD(1)$.
\item For any $D \in \cD$, we have $D \in \cD(1)$ iff $\E \abr{C,D}
  \leq 1$ for any $C \in \cC(1)$.
\end{enumerate}
To avoid trivialities we also assume
\begin{enumerate}
\item[3.] $\cC(1) \supset \cbr{\mathbf{1}}$ where $\mathbf{1} \in \cC$
  denotes the control with $\mathbf{1}_0(\omega)\set 0$ and
  $\mathbf{1}_t(\omega) \set 1$, $t\in (0,\infty]$, $\omega\in\Omega$.
\item[4.] $\cD(1) \not= \cbr{\mathbf{0}}$ where $\mathbf{0} \in
  \cD$ is the trivial state-price deflator given by
  $\mathbf{0}_t(\omega)\set 0$, $t \in [0,\infty]$, $\omega \in \Omega$.
\end{enumerate}

The set $\cC(1)$ will play the role of the budget set for wealth $x =
1$ and $\cD(1)$ can be viewed as a set of state price deflators $D \in
\cD$ (induced, e.g., by a financial market model) for which, in
particular, $\E D_0=\E \abr{\mathbf{1},D} \leq y = 1$.

To formulate the abstract utility maximization problem and its dual
let us put
\begin{equation*}
  \cC(x) \set x \cC(1) \text{ for $x>0$ and } \cD(y) \set y \cD(1)
  \text{ for $y>0$}.
\end{equation*}
It is clear that $\cC(x)$ and $\cD(y)$ inherit the polar relation from
$\cC(1)$ and $\cD(1)$ for any $x,y>0$. By this relation it is also
obvious that these sets are convex and solid (i.e., e.g., with $C \in
\cC(x)$, any $\tilde{C} \in \cC$ with $\tilde{C} \leq C$ is also
contained in $\cC(x)$). Moreover, the lower-semicontinuity of the
pairing $\E\abr{C,D}$, see Lemma~\ref{lem:11}, ensures that $\cC(x)$
and $\cD(y)$ are closed with respect to convergence in the
metric~\eqref{eq:2}.

Finally, let us introduce the value functions
\begin{equation}
  \label{eq:19}
  \mathbf{u}(x) \set \sup_{C \in \cC(x)} \UU(C) \;,\quad x > 0,\  
\end{equation}
and
\begin{equation}
  \label{eq:20}
  \mathbf{v}(y) \set \inf_{D \in \cD(y)} \VV(D) \;, \quad y>0.  
\end{equation}

\begin{Theorem}\label{thm:2}
  Suppose that Assumption~\ref{asp:1} holds true and assume that
  $\mathbf{u}(x)<\infty$ for some $x>0$. Then we have:
\begin{enumerate}
\item The value functions $\mathbf{u}$ of~\eqref{eq:19} and
  $\mathbf{v} $ of~\eqref{eq:20} are real-valued and conjugate to each
  other in the sense that
  \begin{equation}
    \label{eq:21}
    \mathbf{u}(x) = \inf_{y>0} \{\mathbf{v}(y)+xy\}  \text{ for any } x>0
  \end{equation}
  and
  \begin{equation}
    \label{eq:22}
    \mathbf{v}(y) =
    \sup_{x > 0} \{\mathbf{u}(x)-xy\} \text{ for any } y>0\,.
  \end{equation}
  Moreover, $\mathbf{u}$ and $\mathbf{v}$ are continuously
  differentiable on $(0,\infty)$ and satisfy the Inada conditions
  \begin{equation}
    \label{eq:23}
    \mathbf{u}'(0)=\infty, \;     \mathbf{u}'(\infty)=0, \;
    \mathbf{v}'(0)=-\infty, \;     \mathbf{v}'(\infty)=0.    
  \end{equation}
  In addition, $\mathbf{u}$ and $\mathbf{v}$ are, respectively, strictly concave and strictly
  convex, and $y$ attains the infimum in~\eqref{eq:21} iff $x$ attains
  the supremum in~\eqref{eq:22} which in turn is equivalent to both
  \begin{equation}
    \label{eq:24}
    \mathbf{u}'(x)=y \text{ and } \mathbf{v}'(y)=-x\,.
  \end{equation}

\item The infimum in the dual problem~\eqref{eq:20} is attained for
  any $y>0$. All the minimizers $D$ of~\eqref{eq:20} have the same
  envelope process $\breve{D}^y \in \cD(y)$ with~\eqref{eq:12}
  and~\eqref{eq:13}, and, for $x$ given by~\eqref{eq:24},
  \begin{equation}
    \label{eq:25}
    C^x = -V'(-\bdot{\breve{D}}^y) \in \cC(x)
  \end{equation}
  attains the supremum in the primal problem~\eqref{eq:19}.

\item The supremum in the primal problem~\eqref{eq:19} is attained for
  any $x>0$ at a unique $C^x \in \cC(x)$ and, for $y$ given
  by~\eqref{eq:24},
  \begin{equation}
    \label{eq:26}
    \bdot{\breve{D}}^y = -U'(C^x)
  \end{equation}
  yields via~\eqref{eq:12} a $\breve{D}^y \in \cD(y)$ which attains
  the infimum in the dual problem~\eqref{eq:20}.
\end{enumerate}
\end{Theorem}

\section{Illustrations}
\label{sec:illustrations}

Let us illustrate the usefulness of Theorems~\ref{thm:1}
and~\ref{thm:2} by showing how they can be brought to bear on the
classical problems of irreversible investment and of optimal
consumption and investment.

\subsection{Irreversible investment}
\label{sec:irreversibleInvestment}

Consider the manager of a firm who can decide at any point in time $t
\geq 0$ whether or not to expand the currently installed capacity of
production $C_t$. Assuming that installed capacity cannot be reduced
in a profitable way amounts to the assumption that $C \in \cC$ as
introduced in Section~\ref{sec:controls}. Let us suppose that the
revenues $R^C_t$ from the firm's production are an increasing function
of installed capacity and exhibit decreasing returns to
scale. Plainly, it is perfectly reasonable to assume that revenues
also depend on the product's price fluctuations and possibly other
stochastically evolving market conditions. It thus makes sense to
assume that, at time $t \geq 0$, the revenues from a capacity
expansion policy $C \in \cC$ are given as
\begin{equation*}
  R^C_t(\omega) = U_t(\omega,C_t(\omega))
\end{equation*}
for some function $U:\Omega \times [0,\infty) \times [0,\infty) \to
[0,\infty)$ as considered in Assumption~\ref{asp:1}. The manager
discounts future cash flows at some rate $r=(r_t)_{t \geq 0}$, an
optional process with $\int_0^t |r_s|\,ds<\infty$, $t \geq 0$, which
we assume to be such that the random measure
\begin{equation*}
  \mu(dt) \set e^{-\int_0^t r_s\,ds}dt
\end{equation*}
has finite expected mass $\E \mu(0,\infty)<\infty$. 

The expected total discounted revenue is then given by
\begin{equation*}
\E\int_0^\infty e^{-\int_0^t r_s\,ds}R^C_t \,dt= \E \int_0^\infty U_t(C_t) \,d\mu_t=  \UU(C) 
\end{equation*}
exactly as considered in~\eqref{eq:4}. If we now assume that the
(discounted) cost of expanding production capacity by one unit at time
is described by a class (D) supermartingale $Z \geq 0$ with
$Z_\infty=0$ we are led to consider the manager's optimization
problem:
\begin{equation}
  \label{eq:27}
  \text{Maximize } \UU(C) - \E\int_0^\infty Z_t \,dC_t \text{ subject to } C \in \cC.
\end{equation}
This kind of singular control problem is of great interest in
Economics. We refer to~\citet{Alvarez:11} for a more extensive account
of the pertaining literature.

Recalling the Doob-Meyer decomposition $Z=M-A$ into a uniformly
integrable martingale $M$ and a predictable increasing process $A$
with $A_0=0$, we find that $\hat{D} \set M_\infty-A$ is contained in $\cD$
and satisfies
\begin{equation*}
  \E\int_0^\infty Z_t \,dC_t = \E\int_0^\infty\opt{\left(M_\infty-A_t\right)} \,dC_t =
  \E\abr{C,\hat{D}}, \quad C \in \cC.
\end{equation*}
By Theorem~\ref{thm:1}, the value of problem~\eqref{eq:27} is thus
given by the dual functional $\VV(\hat{D})$ of~\eqref{eq:16} and, if it is
finite, we obtain that the optimal capacity expansion plan is
$\hat{C}$ with~\eqref{eq:18}. In particular, an explicit solution
to~\eqref{eq:27} can be given whenever the envelope process
$\breve{D}$ associated with $\hat{D}$ can be computed explicitly. We
refer to \citet{Ferrari:14, Ferrari:14b, BankRiedel01, BankBaumgarten:10} for
such examples.

\subsection{Hindy-Huang-Kreps utility}
\label{sec:HHK}

Following the seminal work of \citet{Mert:71}, the problem of optimal
investment and consumption in continuous-time is mostly studied for
utility functions which depend on the current consumption rate. This
modeling approach was shown by Hindy, Huang, and Kreps (see
\cite{HindyHuangKreps92, HindyHuang92, HindyHuang93}) to fail to
exhibit the economically desirable property of intertemporal
substitution: in Merton's setting, slight shifts in the timing of
consumption plans may lead to significant changes in the utility
associated with these plans. As a remedy, these authors proposed to
consider functionals where utility is derived from a level of
satisfaction, i.e., a weighted average of past consumption such as
\begin{equation*}
  Y^{\tilde{C}}_t \set \int_0^t e^{-\int_s^t\beta_u \,du} \,d\tilde{C}_s, \quad t \geq 0, 
\end{equation*}
where $\tilde{C} \in \cC$ describes the cumulative consumption and
where the locally Lebesgue-integrable adapted process $\beta \geq 0$
measures the decay rate of satisfaction. The utility functional to be
maximized is then
\begin{equation*}
  \tilde{\UU}(\tilde{C}) \set \E \int_0^\infty \tilde{U}(Y^{\tilde{C}}_t) \,d\mu_t
\end{equation*}
where $\tilde{U}:[0,\infty) \to \RR$ is a strictly concave and
increasing utility function of class $C^1$ satisfying the Inada
conditions $\tilde{U}'(0)=\infty$ and $\tilde{U}'(\infty)=0$;
$\mu$, as before, describes an agent's time-preferences and could, for
instance, be specified as $\mu(dt)=e^{-\delta t}dt$ with $\delta>0$.

As usual, the set of consumption plans at the agent's disposal is
determined by his investment opportunities.  Assuming the mild
assumption of no free lunch with vanishing risk we obtain from the
celebrated Fundamental Theorem of Asset Pricing of
\citet{DelbSch:94,DelbSch:98} in great generality that this set can be
described in the form
\begin{equation}\label{eq:28}
\tilde{\cC}(x) = \cbr{\tilde{C} \in \cC \;:\; \E \int_{[0,\infty)} Z_t
  \,dC_t \leq x \text{ for all } Z \in \cZ },
\end{equation}
where $x$ denotes the available initial capital and $\cZ$ denotes a
nonempty set of local martingale deflators, i.e., of
$\P$-supermartingales $Z>0$ with $Z_0=1$ such that for any wealth
process $V$ of an admissible investment strategy the process $ZV$ is a
$\P$-supermartingale.

The agent's optimization problem is then to 
\begin{equation}\label{eq:29}
\text{Maximize } \tilde{\UU}(\tilde{C}) \set \E \int_0^\infty
\tilde{U}(Y^{\tilde{C}}_t) \,d\mu_t \text{ subject to } \tilde{C} \in \tilde{\cC}(x).
\end{equation}
To transform this into the type of utility maximization treated by
our main results in Section~\ref{sec:results} consider the bijection
\begin{equation}\label{eq:30}
  \cC \ni \tilde{C} \mapsto C\set \left(\int_0^t e^{\int_0^s \beta_u \,du}
  \,d\tilde{C}_s\right)_{t \geq 0} \in \cC
\end{equation}
and let
\begin{equation*}
U_t(\omega,c) \set \tilde{U}(e^{-\int_0^t \beta_u(\omega) \,du}c).
\end{equation*}
Then the utility functional $\UU$ of~\eqref{eq:4} satisfies
\begin{equation*}
\UU(C)=\tilde{\UU}(\tilde{C}).
\end{equation*}
Let us also put
\begin{equation*}
  \cC(1) \set \cbr{C \in \cC \;: \tilde{C} \text{ with~\eqref{eq:30}
    is contained in } \tilde{\cC}(1)}
\end{equation*}
and consider its polar
\begin{equation*}
  \cD(1) \set \cbr{D \in \cD \;:\; \E \abr{C,D} \leq 1 \text{ for all }
    C \in \cC(1)}.
\end{equation*}
This latter set is different from $\cbr{\mathbf{0}}$. Indeed, take any
local martingale deflator $Z \in \cZ$ and let $Z=M\tilde{D}$ be its
multiplicative Doob-Meyer decomposition into a local martingale $M$
and a predictable decreasing process $\tilde{D}$ with $\tilde{D}_0=1$.
Let $(T^n)_{n=1,2,\dots}$ be a localizing sequence of stopping times
such that each of the stopped supermartingales $Z^{T^n}$ (and, thus,
each of the stopped local martingales $M^{T^n}$), $n=1,2,\dots,$ is
of class (D). Observe then that $D^n_t \set (M_{T^n}\tilde{D}_t) e^{-\int_0^t \beta_u\,du}
1_{[0,T^n)}(t)$, $t \geq 0$, is contained in $\cD$ and
\begin{equation}\label{eq:31}
  \E\abr{C,D^n} = \E \int_{[0,T^n)}\opt{(M_{T^n}\tilde{D} 
  e^{-\int_0^.  \beta_u \,du})}_t\,dC_t = \E \int_{[0,T^n)} Z_t \,d\tilde{C}_t \leq 1
\end{equation}
for any $C \in \cC(1)$. Hence, $D^n \in \cD(1)$ for each $n=1,2,\dots$.
In fact, letting $n \uparrow \infty$ in~\eqref{eq:31} we
find in conjunction with~\eqref{eq:28}:
\begin{equation*}
  \cC(1) = \cbr{C \in \cC \;:\; \E \abr{C,D} \leq 1 \text{ for any } D
    \in \cD(1)}.
\end{equation*}
Hence, $\cC(1)$ and $\cD(1)$ exhibit the polar relations assumed in
the beginning of Section~\ref{sec:abstractDuality}.

It thus follows that we have the convex duality results of
Theorem~\ref{thm:2} for the Hindy-Huang-Kreps-utility maximization
problem~\eqref{eq:29}. This generalizes the treatment of the complete
market case in~\cite{BankRiedel01} to incomplete market models driven
by general semimartingales and thus also complements the dynamic
programming approach for exponential Levy models with constant
relative risk aversion of~\citet{BenthKarlsenReikvam01}. In
particular, the present paper develops convex duality for optimal
consumption with Hindy-Huang-Kreps preferences at a level of
generality similar to \citet{KramSch:99} for utility from terminal
wealth and to \citet{KaratZitk:03} for utility from the rate of
consumption.

\begin{Remark}\label{rem:1}
  It may be worthwhile to observe that our results also cover the
  finite time horizon case where $\mu$ has support $[0,T]$ for some
  possibly finite stopping time $T>0$. Indeed, in that case we can
  instead consider $\bar{\mu}(dt) \set \mu(dt) +
  1_{(T,\infty)}(t)e^{-t}\,dt$, $\bar{U}_t(c) \set
  1_{[0,T]}(t)U(c)+1_{(T,\infty)}(t) U^*(c)$, where $U^*:[0,\infty)
  \to \RR$ is any deterministic utility function satisfying the Inada
  conditions and having an upper bound $U^*(\infty)<\infty$. The
  budget set will be described by $$\bar{\cD}(1) \set \cbr{D 1_{[0,T)}
    \;:\; D \in \cD(1)}$$ and 
 \begin{align*}
   \bar{\cC}(1) &\set \cbr{C \in \cC \;:\; \E\abr{C,D} \leq 1 \text{
       for all } D \in \bar{\cD}(1)}\\&=\cbr{C \in \cC \;:\; (C_{t
       \wedge T})_{t \geq 0} \in \cC(1)}.  
 \end{align*}
 Then $\bar{U}$, $\bar{\mu}$ satisfy Assumption~\ref{asp:1} if $U$
 does and if $U^*$ has asymptotic elasticity less than one. Moreover,
 $\bar{\cC}(1)$, $\bar{\cD}(1)$ are polar to each other as requested
 in Section~\ref{sec:abstractDuality} and the consumption plans $C,
 \bar{C} \in \cC$ maximizing
\begin{equation*}
  \E \int_0^T U(C_t) \,d\mu_t, \quad \text{respectively } \E \int_0^\infty \bar{U}(\bar{C}_t)\,d\mu_t
\end{equation*}
 subject to $C \in \cC(x)$, respectively, $\bar{C} \in \bar{\cC}(x)$ are
 actually the same up to time $T$ (when all the optimal $\bar{C}$ jump
 to $+\infty$).
\end{Remark}

\section{Proofs of the main results}
\label{sec:proofs}

\subsection{Proof of Theorem~\ref{thm:1}}

Theorem~\ref{thm:1} follows readily from Lemmas~\ref{lem:2}
and~\ref{lem:3} below. These results rely heavily on the following
observation:

\begin{Lemma}\label{lem:1}
  Suppose Assumption~\ref{asp:1} holds true. For $D \in \cD$ let
  $\breve{D}$ denote its envelope with~\eqref{eq:12} and~\eqref{eq:13}
  and, recalling~\eqref{eq:14}, consider $\delta^D \in \bdot{\cD}(D)$
  of~\eqref{eq:11} with
  \begin{equation}
    \label{eq:32}
    C^D \set -V'(\delta^D) \in \cC.
  \end{equation}
  Then $\delta^D$ attains the infimum in the definition~\eqref{eq:10}
  of $\VV(D)$ and, if $\VV(D)<\infty$, $\delta^D$ is in fact the
  unique minimizer in $\bdot{\cD}(D)$, up to modifications on a
  $\P\otimes\mu$-null set.
\end{Lemma}
\begin{proof}
 It is immediate from~\eqref{eq:13} that indeed $\delta^D \in
 \bdot{\cD}(D)$. Uniqueness of minimizers for~\eqref{eq:10} is due to the
 strict convexity of $V$. It thus remains to prove optimality of
 $\delta^D$ for~\eqref{eq:10}. For this it suffices to show that, for $n=1,2,\dots$,
  \begin{equation}
    \label{eq:33}
    \E \int_0^\infty V_n(\delta) \,d\mu \geq \E \int_0^\infty
    V_n(\delta^D) \,d\mu \text{ for any } \delta \in \bdot{\cD}(D)
  \end{equation}
 where 
 \begin{equation}
   \label{eq:34}
   V_n(d) \set \sup_{0 \leq c \leq n} \cbr{U(c)-cd} = 
  \begin{cases} 
   U(n)-n d, & 0 \leq d \leq U'(n),\\
   V(d), & d \geq U'(n).
 \end{cases}
  \end{equation}
  Indeed, it is readily checked that $V_n \geq 0$ is continuously
  differentiable, decreasing and convex on $(0,\infty)$ with $V_n
  \nearrow V$ as $n\uparrow\infty$. Hence, due to monotone
  integration, optimality of $\delta^D$ in~\eqref{eq:10} will follow by
  letting $n \uparrow \infty$ in~\eqref{eq:33}.

 To prove this inequality, we first observe that, by definition and
 convexity of $V_n$,
 \begin{displaymath}
   U(n) = V_n(0) \geq V_n(\delta^D) -V'_n(\delta^D)\delta^D.
 \end{displaymath}
 By Assumption~\ref{asp:1}, $U(n)$ is $\P \otimes
 \mu$-integrable. Since $V_n$, $\delta^D$ and $-V'_n$ are nonnegative,
 it thus follows that also
 \begin{equation}
   \label{eq:35}
   -V'_n(\delta^D)\delta^D= 
   (C^D\wedge  n)\delta^D \in \mathbf{L}^1(\P\otimes \mu),
 \end{equation}
 where the identity is due to the definition~\eqref{eq:32} of $C^D$.

 Again by convexity of $V_n$, we have
 \begin{equation}
   \label{eq:36} V_n(\delta) - V_n(\delta^D) \geq
   V_n'(\delta^D)(\delta-\delta^D) = (C^D \wedge n)\delta^D-(C^D \wedge
   n)\delta.
 \end{equation} 
 So to obtain~\eqref{eq:33}, we have to show that the integral of the
 right side of~\eqref{eq:36} with respect to $\P \otimes \mu$ is
 nonnegative.  To this end, note that
 \begin{displaymath}
   \E \int_0^\infty (C^D \wedge n) \delta \,d\mu = \E
   \abr{C^D \wedge n,\opt{\int_.^\infty \delta \,d\mu}} \leq \E\abr{C^D
     \wedge n,D},
 \end{displaymath} 
 where the last estimate is immediate from $\delta \in
 \bdot{\cD}(D)$. When repeating this calculation for $\delta^D$ instead
 of $\delta$ this estimate turns into an identity because
 of~\eqref{eq:13} and $\cbr{d(C^D\wedge n)>0} \subset \cbr{dC^D>0}$.

 In conjunction with~\eqref{eq:35}, it follows that indeed
 \begin{displaymath} 
   \E \int_0^\infty (C^D \wedge n) \delta \,d\mu
   \leq \E \int_0^\infty (C^D \wedge n) \delta^D \,d\mu<\infty.
 \end{displaymath} 
 This accomplishes our proof.
\end{proof}

\begin{Lemma}\label{lem:2}
  Suppose Assumption~\ref{asp:1} holds true. Then the conjugacy
  relation~\eqref{eq:15} holds. Moreover, if $\UU(\hat{C})<\infty$,
  the infimum in~\eqref{eq:15} is attained for $D \in\cD$ if and only
  if its envelope process with~\eqref{eq:12} and~\eqref{eq:13} is
  actually $\breve{D}=\int_.^\infty U'(\hat{C}) \,d\mu$.
\end{Lemma}
\begin{proof}
 To prove ``$\leq$'' in~\eqref{eq:15}, take $D \in \cD$ with
 $\VV(D)<\infty$ and $\E \abr{\hat{C},D}<\infty$. By Lemma~\ref{lem:1}
 there is $\delta^D \in \bdot{\cD}(D)$ such that $\VV(D)=\E \int_0^\infty
 V(\delta^D) \,d\mu$. Then
 \begin{displaymath}
   \E \int_0^\infty \hat{C} \delta^D \, d\mu = \E \abr{\hat{C},\opt{\int_.^\infty
       \delta^D \,d\mu}} \leq \E \abr{\hat{C},D}<\infty.
 \end{displaymath}
 Thus we can integrate the inequality
 \begin{displaymath}
   0 \leq U(\hat{C}) \leq V(\delta^D)+\hat{C} \delta^D
 \end{displaymath}
 with respect to $\P \otimes \mu$ to deduce that indeed
 \begin{align*}
   0 \leq \UU(\hat{C})  = \E \int_0^\infty U(\hat{C}) \,d\mu 
     &\leq \E
   \int_0^\infty V(\delta^D)\,d\mu+\E\int_0^\infty \hat{C} \delta^D
   \,d\mu \\&\leq \VV(D)+\E\abr{\hat{C},D}
 \end{align*}

 For ``$\geq$'' in~\eqref{eq:15} we can assume $\UU(\hat{C})=\E
 \int_0^\infty U(\hat{C})\,d\mu<\infty$ without loss of generality. Let
 $\hat{\delta} \set U'(\hat{C})$ and note that because $U$ is concave
 in $c$ with $U(0)=0$ we have
 \begin{equation}\label{eq:37}
  0 \leq \hat{C}\hat{\delta} = \hat{C} U'(\hat{C}) \leq U(\hat{C}) \in
  \mathbf{L}^1(\P \otimes \mu).
\end{equation}
 Moreover, $\hat{D} \set \int_.^\infty \hat{\delta} \,d\mu
 \in \cD$ satisfies
 \begin{align*}
   \VV(\hat{D}) \leq \E \int_0^\infty V(\hat{\delta}) \,d\mu = \E
   \int_0^\infty U(\hat{C}) \,d\mu+\E \int_0^\infty
   \hat{C}\hat{\delta} \,d\mu<\infty.
 \end{align*}
 From Lemma~\ref{lem:1} it now follows that in fact $\VV(\hat{D}) = \E
 \int_0^\infty V(\hat{\delta}) \,d\mu<\infty$. We thus can integrate
 the identity
 \begin{align*}
   U(\hat{C}) = V(\hat{\delta})-\hat{C}\hat{\delta}
 \end{align*}
 with respect to $\P \otimes \mu$ to obtain 
 \begin{displaymath}
   \UU(\hat{C}) = \VV(\hat{D})-\E \int_0^\infty \hat{C} \hat{\delta}
   \,d\mu
   =\VV(\hat{D})-\E\abr{\hat{C},\hat{D}}.
 \end{displaymath}
 This gives ``$\geq$'' in~\eqref{eq:15}.

 The preceding argument already establishes the ``if''-part of the
 present lemma. For the ``only if''-part assume that $D \in \cD$
 satisfies $\UU(\hat{C})=\VV(D)+\E \abr{\hat{C},D}<\infty$. Clearly,
 we have $\VV(D)<\infty$ then. Thus, by Lemma~\ref{lem:1}, there is
 $\delta^D \in \bdot{\cD}(D)$ with $\VV(D)=\E\int_0^\infty
 V(\delta^D)\,d\mu<\infty$. Moreover, the choice of $D$ entails $\E
 \int_0^\infty \hat{C} \delta^D \,d\mu \leq \E\abr{\hat{C},D}<\infty$.
 Now, integrating
 \begin{equation}
   \label{eq:38}
   U(\hat{C}) \leq V(\delta^D)+\hat{C}\delta^D
 \end{equation}
 we find 
 \begin{equation}\label{eq:39}
   \UU(\hat{C}) \leq \VV(D)+\E\abr{\hat{C},\opt{\int_.^\infty \delta^D
       \,d\mu}} \leq \VV(D)+\E \abr{\hat{C},D}=\UU(\hat{C})<\infty.
 \end{equation}
 So, equality must hold true in all the above estimates. It follows
 that equality holds $\P \otimes \mu$-almost everywhere
 in~\eqref{eq:38} which readily implies $\delta^D=U'(\hat{C})$
 $\P\otimes\mu$-almost everywhere, and, thus, $\opt{\int_.^\infty
   U'(\hat{C}) \,d\mu}=\opt{\int_.^\infty \delta^D \,d\mu} \leq
 \opt{D}$. Moreover, \eqref{eq:39} then also yields $\E
 \abr{\hat{C},\opt{\int_.^\infty U'(\hat{C})
     \,d\mu}}=\E\abr{\hat{C},\opt{D}}$, i.e., in fact
 $\opt{\int_.^\infty U'(\hat{C}) \,d\mu}=\opt{D}$ on
 $\cbr{d\hat{C}>0}$. By Lemma~\ref{lem:1}, this identifies
 $\int_.^\infty U'(\hat{C})\,d\mu$ as the envelope process $\breve{D}$
 of $D$ with~\eqref{eq:12} and~\eqref{eq:13}.  This accomplishes our
 proof. 
\end{proof}

\begin{Lemma}\label{lem:3}
  Suppose Assumption~\ref{asp:1} holds. Then the conjugacy relation~\eqref{eq:16} holds. Moreover, if
 $\VV(\hat{D})<\infty$, the supremum in~\eqref{eq:16} is attained
 exactly for $\hat{C} \set C^{\hat{D}}$ where $C^{\hat{D}}$ is defined
 in Lemma~\ref{lem:1}.
\end{Lemma}
\begin{proof}
  Let us first apply Lemma~\ref{lem:1} to obtain that
  there is $\hat{\delta} \set \delta^{\hat{D}} \in \bdot{\cD}(\hat{D})$
  with $\VV(\hat{D})=\E\int_0^\infty V(\hat{\delta})\,d\mu$. 

  To see that ``$\geq$'' holds in~\eqref{eq:16}, take $C \in \cC$ with
  $\UU(C)=\E \int_0^\infty U(C) \,d\mu<\infty$. Without loss of
  generality we can assume $\VV(\hat{D})<\infty$ and
  $\E\abr{C,\hat{D}}<\infty$. Then all terms in the inequality
  \begin{displaymath}
    V(\hat{\delta}) \geq U(C)-C\hat{\delta}
  \end{displaymath}
  are $\P\otimes\mu$-integrable. Upon integration we get
  $\VV(\hat{D}) \geq \UU(C)-\E\abr{C,\int_.^\infty
    \hat{\delta}\,d\mu}$. This implies the desired estimate since
  $\opt{\left(\int_.^\infty \hat{\delta}\,d\mu\right)}\leq\opt{\hat{D}}$.

  For the proof of ``$\leq$'' in~\eqref{eq:16} consider $\hat{C} \set
  -V'(\hat{\delta}) \in \cC$ where $\hat{\delta}$ is chosen as
  above. If $\VV(\hat{D})=\infty$, we consider $C^n \set \hat{C}\wedge n \in
  \cbr{\UU<\infty}$, $n=1,2,\dots$, in~\eqref{eq:16} to deduce
  \begin{displaymath}
    \UU(C^n)-\E\abr{C^n,\hat{D}} = \E\int_0^\infty (U(\hat{C}
      \wedge n)-(\hat{C} \wedge n) U'(\hat{C}\wedge n)\,d\mu = \E
    \int_0^\infty V_n(\hat{\delta})\,d\mu
  \end{displaymath}
  where $V_n$ is as in~\eqref{eq:34}. Since $V_n \nearrow V$, it
  follows by monotone integration that as $n\uparrow\infty$ the above
  expression converges to $\E\int_0^\infty V(\hat{\delta})\,d\mu \geq
  \VV(\hat{D})$ and we obtain ``$\leq$'' in~\eqref{eq:16} in case
  $\VV(\hat{D})=\infty$.

  For the remaining case where $\E\int_0^\infty V(\hat{\delta})\,d\mu<\infty$,
  let us first show that $\UU(\hat{C})=\E\int_0^\infty
  U(\hat{C})\,d\mu<\infty$. Indeed, by Assumption~\ref{asp:1} the
  asymptotic elasticity of $U$ is uniformly less than one in the sense
  that $c U'(c) < \gamma U(c)$ for $c > C^\gamma$ where $\gamma \in
  [0,1)$. Thus, we have
 \begin{displaymath}
   \mathbf{L}^1(\P\otimes\mu) \ni V(\hat{\delta})
  =U(\hat{C})-\hat{C}U'(\hat{C}) 
  \geq (1-\gamma) U(\hat{C}) \geq 0 
  \text{ on} \cbr{\hat{C}>C^\gamma}.
 \end{displaymath}
 Since by assumption $\E \int_0^\infty U(C^\gamma)\,d\mu<\infty$, it
 thus follows that $U(\hat{C}) \in \mathbf{L}^1(\P\otimes \mu)$, i.e.,
 $\UU(\hat{C})<\infty$.

 Now, recalling the estimate~\eqref{eq:37}, we deduce from
 $\UU(\hat{C})<\infty$ that also $\E\abr{\hat{C},\hat{D}}=\E\int_0^\infty
 \hat{C}\hat{\delta}\,d\mu<\infty$. The ``$\leq$''-claim now follows
 upon integration of $V(\hat{\delta})=U(\hat{C})-\hat{C}\hat{\delta}$
 with respect to $\P\otimes\mu$. This also establishes the ``if''-part
 of our lemma. The ``only if''-part follows immediately from this and
 the strict concavity of $\UU$ on $\cbr{\UU<\infty}$ which implies the
 uniqueness of the optimizer~$\hat{C}$.
\end{proof}

\subsection{Proof of Theorem~\ref{thm:2}}

The proof of Theorem~\ref{thm:2} is prepared by the following
Lemmas~\ref{lem:4}--\ref{lem:9}.

\begin{Lemma}\label{lem:4}
  Under the assumptions of Theorem~\ref{thm:2}, we have
  \begin{equation}
   \label{eq:40}
   \mathbf{v}(y) = \inf_{\delta \in \bdot{\mathcal{D}}(y)} \E \int_0^\infty
     V(\delta) \,d\mu, \quad y>0,
  \end{equation}
  where 
  \begin{equation}
    \label{eq:41}
    \bdot{\cD}(y) \set \bigcup_{D \in \cD(y)} \bdot{\cD}(D).
  \end{equation}
  Moreover, for any $y>0$ with $\mathbf{v}(y)<\infty$, the infimum
  in~\eqref{eq:40} is attained at a unique $\delta^y \in
  \bdot{\mathcal{D}}(y)$ for which, in addition, $\hat{C}^y \set
  -V'(\delta^y)$ is contained in $\cC$. Finally, $\mathbf{v}$ is
  strictly convex on $\cbr{\mathbf{v}<\infty}$.
\end{Lemma}
\begin{proof}
 Identity~\eqref{eq:40} is immediate from~\eqref{eq:41} and Lemma~\ref{lem:1}.

 Now assume $\mathbf{v}(y)<\infty$ and consider a minimizing sequence
 $\delta^n \in \bdot{\cD}(y)$ for~\eqref{eq:40}. By Lemma~A1.1 of
 \citet{DelbSch:94} there is a sequence $\tilde{\delta}^n$
 of convex combinations of $\delta^n, \delta^{n+1},\dots$ which
 converges $\P \otimes \mu$-almost everywhere to an optional
 $\delta^y$ taking values in $[0,\infty]$. In fact, $\delta^y \in
 \bdot{\cD}(y)$ because $D^y \set \int_.^\infty \delta^y \,d\mu \in
 \cD(y)$, which holds since by Fatou's lemma
 \begin{align*}
   \E \abr{C,D^y} = \E\int_0^\infty C \delta^y \,d\mu &\leq \liminf_{n}
   \E\int_0^\infty C\tilde{\delta}^n \,d\mu\\&=\liminf_n \E\abr{C,\int_.^\infty
     \tilde{\delta}^n\,d\mu} \leq xy
 \end{align*}
 for any $C \in \cC(x)$, $x>0$. Here the last inequality follows because
 $\tilde{\delta}^n \in \bdot{\cD}(y)$ by convexity of this set. Another
 application of Fatou's lemma reveals
 \begin{displaymath}
   \E \int_0^\infty V(\delta^y)\,d\mu \leq \liminf_n \E\int_0^\infty
   V(\tilde{\delta}^n)\,d\mu \leq \liminf_n \E\int_0^\infty V(\delta^n)\,d\mu=\mathbf{v}(y)
 \end{displaymath}
 by convexity of $V$ and our choice of $(\delta^n)_{n=1,2,\dots}$ as a
 minimizing sequence. This proves existence of a minimizer
 for~\eqref{eq:40}. Uniqueness up to a $\P \otimes \mu$-null set
 follows from the strict convexity of $V$. In fact, applying
 Lemma~\ref{lem:1} for $D \set D^y$ reveals that $\delta^y$ has a
 predictable $\P\otimes \mu$-modification which is unique up to
 indistinguishability if we require, in addition, that $-V'(\delta^y)
 \in \cC$. Strict convexity of $\mathbf{v}$ on
 $\cbr{\mathbf{v}<\infty}$ now follows from strict convexity and
 strict monotonicity of $V$.
\end{proof}

\begin{Lemma}\label{lem:5}
  Under the assumptions of Theorem~\ref{thm:2}, the primal value function
  $\mathbf{u}$ of~\eqref{eq:19} is real-valued and conjugate to the
  dual value function $\mathbf{v}$ of~\eqref{eq:20} in the sense
  that~\eqref{eq:21} and~\eqref{eq:22} hold true.
\end{Lemma}
\begin{proof}
  The primal value function $\mathbf{u}$ is, by assumption, finite at
  some point $x>0$. Its concavity then yields that it is finite and,
  thus, continuous on all of $(0,\infty)$. Therefore, by classical
  duality results (cf., e.g., Theorem~12.2 in \citet{Rock:70}),
  \eqref{eq:21} follows from~\eqref{eq:22}. 

  Let us first argue that ``$\geq$'' holds in~\eqref{eq:22}. So take
  $C \in \cC(x)$ and $D \in \cD(y)$. Then $\E \abr{C,D} \leq xy$ and,
  by equation~\eqref{eq:15} of Theorem~\ref{thm:1}:
  \begin{displaymath}
    \UU(C)-xy \leq \VV(D)+\E\abr{C,D}-xy \leq \VV(D).
  \end{displaymath}
  Taking the supremum over $C \in \cC(x)$ and the infimum over $D \in
  \cD(y)$ in this relation yields ``$\geq$'' in~\eqref{eq:22}.

  To obtain that also ``$\leq$'' holds in~\eqref{eq:22}, we shall
  employ the Minimax Theorem~\ref{thm:3} from the appendix with
  \begin{itemize}
  \item $\cA \set \cC_n \set \cbr{C \in \cC \;:\: C_\infty \leq n}$
    where $n \in \cbr{1,2,\dots}$, a convexly compact subset of the
    space of left-continuous processes with bounded total variation
    endowed with the metric $\dist$ of~\eqref{eq:2}; see
    Lemma~\ref{lem:12}.
 \item $\cB \set \cD(y)$ which can be viewed as a convex, closed
   subset of the space of right-continuous processes with
   $\P$-integrable total variation endowed with convergence with
   respect to the distance $\dist$ of~\eqref{eq:2} because $\E D_0 =\E
   \abr{\mathbf{1},D}\leq y$ by assumption on $\cD(y)=y\cD(1)$; and
   with
 \item $\HH(C,D) \set \UU(C)-\E\abr{C,D}$, which is convex (even
   linear) in $D \in \cB=\cD(y)$ and concave and upper-semicontinuous
   in $C \in \cA =\cC_n$, because, with respect to the metric $\dist$,
   $\UU$ is continuous on $\cC_n$ by dominated convergence and
   $\E\abr{.,D}$ is lower-semicontinuous due to Lemma~\ref{lem:11}.
 \end{itemize}
 We thus obtain that, for $n=1,2,\dots$,
 \begin{align}\label{eq:42}
   \sup_{C \in \cC _n} \inf_{D \in \cD(y)}
   \cbr{\UU(C)-\E\abr{C,D}} = 
  \inf_{D \in \cD(y)} \sup_{C \in \cC_n} \cbr{\UU(C)-\E\abr{C,D}}
 \end{align}

 Let us next prove that, as $n \uparrow \infty$, the left side
 of~\eqref{eq:42} converges to $\sup_{0 \leq x
   <\infty}\cbr{\mathbf{u}(x)-xy}$. Clearly, with $\pi(C) \set \sup_{D
   \in \cD(1)} \E\abr{C,D}$ the limit of left side of~\eqref{eq:42}
 can be written as
 \begin{align*}
    \sup_{C \in \cC  \text{ bounded}} \inf_{D \in \cD(y)}
   \cbr{\UU(C)-\pi(C)y} &
    = \sup_{0 \leq x <\infty} \sup_{C \in \cC(x) \text{ bounded}}
    \cbr{\UU(C)-xy} \\&
   = \sup_{0 \leq x <\infty} \cbr{\mathbf{u}(x)-xy}
 \end{align*}
 where the last identity holds because by monotone convergence 
 $\UU(C)= \lim_n \UU(C\wedge n)$, $C \in \cC$, so that the utility of
 any $C$ can be approximated by the utility of bounded controls.

 Now the proof of the present lemma will be accomplished once we have
 shown that, as $n \uparrow \infty$, the right side of~\eqref{eq:42}
 tends to a limit which is not smaller than $\mathbf{v}(y)$. To this
 end, we first observe that
 \begin{equation}\label{eq:43}
   \sup_{C \in \cC_n}\cbr{\UU(C)-\E\abr{C,D}} 
  = \E\int_0^\infty V_n(-\bdot{\breve{D}})  \,d\mu \text{ for any } D \in\cD(y)
 \end{equation}
 where $V_n$ is given by~\eqref{eq:34}. Indeed, because $\opt{D} \geq
 \opt{\breve{D}}$, we have 
 \begin{displaymath}
   \UU(C)-\E\abr{C,D} \leq
   \UU(C)-\E\abr{C,\breve{D}} = \E\int_0^\infty
   U(C)-C(-\bdot{\breve{D}})\,d\mu,
 \end{displaymath}
 where for $C \in \cC_n$ the last integrand is not larger than
 $V_n(-\bdot{\breve{D}})$. This proves ``$\leq$''
 in~\eqref{eq:43}. For ``$\geq$'' we just need to observe that $C \set
 -V'_n(-\bdot{\breve{D}}) = -V'(-\bdot{\breve{D}}) \wedge n \in \cC_n$
 will give equality in both of the preceding estimates.

 Due to~\eqref{eq:43}, we can take $D^n \in \cD(y)$ with $0 \leq
 \delta^n \set-\bdot{\breve{D}}^n$ such that $\E \int_0^\infty
 V_n(\delta^n)\,d\mu$ converges to the limit of the right side
 of~\eqref{eq:42} as $n \uparrow \infty$.  By Lemma~A1.1
 in~\cite{DelbSch:94} there are $\tilde{\delta}^n \in
 \conv\cbr{\delta^n,\delta^{n+1},\dots}$, $n=1,2,\dots$, which
 converge $\P \otimes \mu$-almost everywhere to some $\delta^* \geq
 0$. Because all $\delta^n$ are contained in $\bdot{\cD}(y)$, so are,
 by convexity of this set, all the $\tilde{\delta}^n$. In fact, also
 $\delta^* \in \bdot{\cD}(y)$ because $D^* \set \int_.^\infty
 \delta^*\,d\mu \in \cD(y)$ as by Fatou's lemma
 \begin{align*}
   \E \abr{C,D^*} 
  = \E \int_0^\infty C\delta^*\,d\mu 
& \leq \liminf_n \E \int_0^\infty C\tilde{\delta}^n\,d\mu\\
& = \liminf_n\E \abr{C,\opt{\int_.^\infty \tilde{\delta}^n\,d\mu}} 
 \leq xy
 \end{align*}
for any $C \in \cC(x)$.

 It follows that for $N=1,2,\dots$:
 \begin{align*}
  \lim_n & \inf_{D \in \cD(y)} \sup_{C \in \cC_n} \cbr{\UU(C)-\E\abr{C,D}} 
  = \lim_n  \E \int_0^\infty V_n(\delta^n)\,d\mu\\
 &\geq \liminf_n  \E \int_0^\infty V_N(\delta^n)\,d\mu
 \geq \liminf_n  \E \int_0^\infty V_N(\tilde{\delta}^n)\,d\mu
 \geq \E \int_0^\infty V_N(\delta^*)\,d\mu\\
 &\longto_{N \uparrow \infty} \E \int_0^\infty V(\delta^*)\,d\mu
 \geq \mathbf{v}(y)
 \end{align*}
 where the first estimate and the convergence follow from $V_n \geq
 V_N \nearrow V$ for $n \geq N \uparrow \infty$. The second estimate
 is due to the convexity of $V_N$ and the third is due to Fatou's
 lemma. The last estimate is immediate from Lemma~\ref{lem:1} and
 $\delta^* \in \bdot{\cD}(y)$.
\end{proof}

\begin{Lemma}\label{lem:6}
  Under the assumptions of Theorem~\ref{thm:2}, $\mathbf{v}$
  of~\eqref{eq:20} is real-valued, strictly convex and strictly
  decreasing on $(0,\infty)$. Moreover, $\mathbf{u}$ of~\eqref{eq:19}
  is continuously differentiable on $(0,\infty)$ with
  $\mathbf{u}'(\infty)=0$.
\end{Lemma}
\begin{proof}
  Let us first show that even
  \begin{equation}
    \label{eq:44}
    \lim_{x \uparrow \infty} \mathbf{u}(x)/x=0.
  \end{equation}
  Indeed, since $\mathbf{u}$ takes real values by Lemma~\ref{lem:5},
  we can find, for $\epsilon>0$ and $x>0$, a $C^{x,\epsilon} \in
  \cC(x)$ such that $\mathbf{u}(x) \leq
  \UU(C^{x,\epsilon})+\epsilon$. Then, by the equivalent
  formulation~\eqref{eq:5} of our asymptotic elasticity condition~\eqref{eq:3},
  \begin{displaymath}
    U(C^{x,\epsilon}) \leq x^{\gamma} U(C^{x,\epsilon}/x) \text{ on }
    \cbr{C^{x,\epsilon} \geq C^\gamma}.
  \end{displaymath}
  Upon integration with respect to $\P \otimes \mu$ we thus obtain
  \begin{align*}
    \mathbf{u}(x) &\leq x^{\gamma} \E \int_0^\infty
    U(C^{x,\epsilon}/x)
    \,d\mu+\E \int_0^\infty U(C^\gamma)\,d\mu+\epsilon\\
    & \leq x^\gamma \mathbf{u}(1)+\E \int_0^\infty U(C^\gamma)\,d\mu+\epsilon,
  \end{align*}
  where we used that $C^{x,\epsilon}/x \in \cC(1)$.  Since $\gamma \in
  [0,1)$, our claim~\eqref{eq:44} now follows upon division by $x
  \uparrow \infty$.

  In conjunction with~\eqref{eq:44}, the duality between $\mathbf{u}$
  and $\mathbf{v}$ established in Lemma~\ref{lem:5} yields that
  $\mathbf{v}(y)<\infty$ for $y>0$.  By Lemma~\ref{lem:1},
  $\mathbf{v}$ is thus strictly convex on $(0,\infty)$. This
  immediately implies that $\mathbf{v}$ is strictly decreasing, By
  classical convex duality results (e.g. \citet{Rock:70}), strict
  convexity of $\mathbf{v}$ implies the differentiability of its
  conjugate $\mathbf{u}$ on $(0,\infty)$. By concavity and
  monotonicity, $0 \leq \mathbf{u}'(x) \leq
  \mathbf{u}(x)/x$. So~\eqref{eq:44} also yields
  $\mathbf{u}'(\infty)=0$.
\end{proof}

The following lemma is a minor adaptation of Lemmas~3.6 and~3.7
in~\citet{KramSch:99}:

\begin{Lemma}\label{lem:7}
  Under the assumptions of Theorem~\ref{thm:2}, the minimizers
  $\delta^y \in \bdot{\cD}(y)$ from Lemma~\ref{lem:4} depend
  continuously on $y>0$ in the sense that the mapping
\begin{displaymath}
  (0,\infty) \ni y \mapsto (\delta^y,
    V(\delta^y),-V'(\delta^y)\delta^y) \in \mathbf{L}^0(\P\otimes\mu) \times
    \mathbf{L}^1(\P\otimes\mu) \times \mathbf{L}^1(\P\otimes\mu)
\end{displaymath}
is continuous.
\end{Lemma}
\begin{proof}
  That the above mapping is indeed defined on all of $(0,\infty)$ is
  due to the finiteness of $\mathbf{v}$ on $(0,\infty)$ established in
  Lemma~\ref{lem:6}. 

  We first prove that $\delta^{y_n} \to \delta^y$ in
  $\mathbf{L}^0(\P\otimes\mu)$ for any $y_n \to y \in (0,\infty)$. If
  $\delta^{y_n}$ does not converge to $\delta^y$ in this sense then
  there is $\epsilon>0$ such that
  \begin{displaymath}
    \limsup_n \P\otimes\mu\left[|\delta^{y_n}-\delta^y|>\epsilon, 
    \delta^{y_n}+\delta^y<1/\epsilon\right]>\epsilon,
  \end{displaymath}
  where we recall that $(\delta^{y_n})_{n=1,2,\dots}$ is bounded in
  $\mathbf{L}^1(\P\otimes\mu)$ because $\E\int_0^\infty \delta^{y_n}
  \,d\mu \leq y_n \to y>0$ by definition of $\bdot{\cD}(y^n)$.
  Observe now that by strict convexity of $V$, $\delta^n \set
  \frac{1}{2}(\delta^{y_n}+\delta^y)$ satisfies
  \begin{displaymath}
  V(\delta^n) \leq \frac{1}{2}(V(\delta^{y_n})+V(\delta^y))  
  \end{displaymath}
  and, for some sufficiently small $\eta>0$, also
  \begin{displaymath}
    \limsup_n \P\otimes\mu\left[V(\delta^n) \leq
      \frac{1}{2}(V(\delta^{y_n})+V(\delta^y)) -\eta\right]>\eta.
  \end{displaymath}
 Upon integration with respect to $\P\otimes\mu$ it follows that
  \begin{align*}
    \limsup_n \E\int_0^\infty V(\delta^n)\,d\mu 
   &\leq \limsup_n \frac{1}{2}\E\int_0^\infty
   \left(V(\delta^{y_n})+V(\delta^y)\right)\,d\mu-\eta^2\\
  &=\limsup_n \frac{1}{2}(\mathbf{v}(y^n)+\mathbf{v}(y))-\eta^2= \mathbf{v}(y)-\eta^2
  \end{align*}
  where the last identity is due to the continuity of the convex
  function $\mathbf{v}$. On the other hand, by the scaling property
  and convexity of the sets $\cD(y)=y\cD(1)$, we have $\delta^n \in
  \bdot{\cD}(y\vee y^n)$ and therefore, by Lemma~\ref{lem:4},
  \begin{align*}
    \mathbf{v}(y) = \lim_n \mathbf{v}(y \vee y^n) \leq \liminf_n
    \E\int_0^\infty V(\delta^n)\,d\mu.
  \end{align*}
  This clearly contradicts the preceding inequality and so  we
  must have indeed that $\delta^{y_n}
  \to \delta^y$ in $\mathbf{L}^0(\P\otimes\mu)$.

  Convergence of $V(\delta^{y_n}) \geq 0$ in $\mathbf{L}^1(\P\otimes\mu)$ now
  follows from convergence in $\mathbf{L}^0(\P\otimes\mu)$ and 
  \begin{displaymath}
    \E
  \int_0^\infty V(\delta^{y_n})\,d\mu=\mathbf{v}(y^n) \longto_{n
    \uparrow \infty} \mathbf{v}(y)=\E
  \int_0^\infty V(\delta^y)\,d\mu.
  \end{displaymath}
  Moreover, $\mathbf{L}^1(\P\otimes\mu)$-convergence of
  $(-V'(\delta^{y_n})\delta^{y_n})_{n=1,2,\dots}$ will follow once we
  have established the uniform $\P\otimes\mu$-integrability of this
  sequence. Our uniform asymptotic elasticity condition~\eqref{eq:7} gives
  \begin{equation}
    \label{eq:45}
    (1-\gamma) (-V'(\delta^{y_n}))\delta^{y_n} \leq \gamma
    V(\delta^{y_n}) \text{ on } \cbr{\delta^{y_n} < D^{\gamma}}
  \end{equation}
  where $\gamma \in (0,1)$ and where $D^{\gamma} \set
  U'(C^{\gamma})$. Moreover, we have, with $C^n \set
  -V'(\delta^{y_n})$, that
  \begin{equation}
    \label{eq:46}
   0\leq (-V'(\delta^{y_n}))\delta^{y_n} = C^n U'(C^n) \leq U(C^n)
   \leq U(C^\gamma)
  \end{equation}
  on $\cbr{\delta^{y_n} \geq D^{\gamma}}=\cbr{C^n \leq C^\gamma}$.  In
  conjunction with the already established
  $\mathbf{L}^1(\P\otimes\mu)$-convergence of
  $(V(\delta^{y_n}))_{n=1,2,\dots}$ and our assumption that
  $U(C^\gamma)$ is $\P \otimes \mu$-integrable, the combination of the
  estimates~\eqref{eq:45} and~\eqref{eq:46} yields the desired
  uniform integrability.
\end{proof}

We now can use a variant of the argument in Lemma~3.8 of
\citet{KramSch:99} to deduce:

\begin{Lemma}\label{lem:8}
  Under the assumptions of Theorem~\ref{thm:2}, the dual value
  function $\mathbf{v}$ is continuously differentiable on $(0,\infty)$
  with
  \begin{equation}
    \label{eq:47}
    \mathbf{v}'(y)y = \E \int_0^\infty V'(\delta^y)\delta^y \,d\mu,
    \quad y>0,
  \end{equation}
  and $\mathbf{v}'(0)=-\infty$, $\mathbf{v}'(\infty)=0$. Moreover,
  $\mathbf{u}$ is strictly increasing and strictly concave on
  $(0,\infty)$ with $\mathbf{u}'(0)=\infty$.
\end{Lemma}
\begin{proof}
  We first observe that for $y>0$ and $\lambda_n \downarrow 1$,
  \begin{equation}
    \label{eq:48}
    (-V'(\delta^{\lambda_n y}/\lambda_n)
  \delta^{\lambda_ny})_{n=1,2,\dots} \text{ is uniformly $\P\otimes\mu$-integbrale.}
  \end{equation}
  Indeed, by employing successively our dual asymptotic elasticity
  estimates~\eqref{eq:7} and~\eqref{eq:8} and also the monotonicity
  of $V$, we get
 \begin{align*}
   0 \leq -V'(\delta^{\lambda_n y}/\lambda_n) \delta^{\lambda_n y} 
  & \leq \frac{\gamma}{1-\gamma} V(\delta^{\lambda_n y}/\lambda_n) \\
  & \leq \frac{\gamma}{1-\gamma} V((\delta^{\lambda_n y} \wedge D^\gamma)/\lambda_n) \\
  & \leq \frac{\gamma}{1-\gamma}
  \left(\frac{1}{\lambda_n}\right)^{\frac{\gamma}{1-\gamma}}
  V(\delta^{\lambda_n y} \wedge D^{\gamma}) \\
  &= \frac{\gamma}{1-\gamma}
  \left(\frac{1}{\lambda_n}\right)^{\frac{\gamma}{1-\gamma}}
  \left(V(\delta^{\lambda_n y}) \vee V(D^{\gamma})\right)
 \end{align*}
 on $\cbr{\delta^{\lambda_n y}/\lambda_n \leq D^\gamma}$. With
 $C^{\lambda_n} \set -V'(\delta^{\lambda_n y}/\lambda_n)$ the
 complement of this set is $\cbr{\delta^{\lambda_n y}/\lambda_n >
   D^\gamma}=\cbr{C^\lambda_n < C^\gamma}$ and so, on this set,
 \begin{equation*}
   0 \leq -V'(\delta^{\lambda_n y}/\lambda_n) \delta^{\lambda_n y}
   =U'(C^{\lambda_n})C^\lambda_n\lambda_n\leq U(C^{\lambda_n})
   \lambda_n \leq U(C^\gamma)\lambda_n.
 \end{equation*}
 Hence, to obtain our claim~\eqref{eq:48} it suffices to observe that
 on either set we find an upper bound which is uniformly
 integrable. This is clear for
 $(U(C^\gamma)\lambda_n)_{n=1,2,\dots}$. On the other hand,
 Lemma~\ref{lem:7} yields in particular the uniform
 $\P\otimes\mu$-integrability of $(V(\delta^{\lambda_n
   y}))_{n=1,2,\dots}$ and finally $V(D^{\gamma}) \leq U(C^\gamma) \in
 \mathbf{L}^1(\P\otimes\mu)$.

  We now can argue exactly as in Lemma~3.8 of \citet{KramSch:99} and
  pass to the limit $\lambda \downarrow 1$ in
  \begin{displaymath}
    \frac{\mathbf{v}(\lambda y)-\mathbf{v}(y)}{\lambda -1}
  \leq \E\int_0^\infty \frac{V(\lambda
    \delta^y)-V(\delta^y)}{\lambda-1} \,d\mu
  \end{displaymath}
  by monotone convergence and, by uniform integrability, also in
  \begin{align*}
    \frac{\mathbf{v}(\lambda y)-\mathbf{v}(y)}{\lambda -1}
  &\geq \E\int_0^\infty \frac{V(
    \delta^{\lambda y})-V(\delta^{\lambda y}/\lambda)}{\lambda-1}
\,d\mu \\&\geq \E \int_0^\infty V'(\delta^{\lambda y}/\lambda)\delta^{\lambda y}\frac{1-1/\lambda}{\lambda-1}\,d\mu
  \end{align*}
 to see, respectively, that $(\mathbf{v}(\lambda
  y)-\mathbf{v}(y))/(\lambda - 1)$ has a $\limsup$ not larger and a
  $\liminf$ not smaller than the right side of~\eqref{eq:47}. The
  continuity of this expression established in Lemma~\ref{lem:7} in
  conjunction with the convexity of $\mathbf{v}$ then implies our
  claim.

  The strict monotonicity of $\mathbf{u}$ now follows from its strict
  concavity on $(0,\infty)$ which, in turn, is equivalent to the
  differentiability of its convex conjugate $\mathbf{v}$ on
  $(0,\infty)$ by classical duality results; see, e.g., Theorem~26.3 in
  \citet{Rock:70}. These same results also yield the equivalence of
  $\mathbf{v}'(\infty)=0$ and $\mathbf{u}'(0)=\infty$, and the first
  of these relations is immediate from the fact that $\mathbf{v}$ is
  strictly convex and decreasing and bounded from below (by $0$).
  Similarly, classical duality yields the equivalence of
  $\mathbf{v}'(0)=-\infty$ and $\mathbf{u}'(\infty)=0$, where the last
  relation was already established in Lemma~\ref{lem:6}.
\end{proof}

\begin{Lemma}\label{lem:9}
  Under the assumptions of Theorem~\ref{thm:2},  $C^x \set
  -V'(\delta^y)$ with $y=\mathbf{u}'(x)$ and $\delta^y$ as in
  Lemma~\ref{lem:4} is the unique control in $\cC(x)$ that
  attains $\mathbf{u}(x)=\UU(C^x)$ for $x>0$.
\end{Lemma}
\begin{proof}
  Uniqueness of a maximizer $C^x$ for $\mathbf{u}(x)$ is immediate
  from the strict concavity of $U$. Define $D^y \set \int_.^\infty
  \delta^y \,d\mu \in \cD(y)$ and observe that due to~\eqref{eq:15} in
  Theorem~\ref{thm:1} we have for any $C \in \cC(x)$:
  \begin{displaymath}
    \UU(C) \leq \VV(D^y)+\E\abr{C,D^y} \leq \VV(D^y)+xy.
  \end{displaymath}
  By item~3. of Theorem~\ref{thm:1} and by~\eqref{eq:47}, we have
  equalities in both of these estimates when we consider
  $C \set C^x=-V'(\delta^y)$. It thus suffices to prove that $C^x \in
  \cC(x)$, i.e., that 
  \begin{equation}
    \label{eq:49}
    \E \abr{C^x,D} \leq xy \text{ for any } D \in \cD(y).
  \end{equation}

  For this, we first note that, for any such $D$, we have
  \begin{equation}
    \label{eq:50}
    D^\epsilon \set
  \epsilon D + (1-\epsilon)D^y \in \cD(y) \text{ with }
  \VV(D^\epsilon)<\infty, \quad 0<\epsilon<1.
  \end{equation}
  Indeed, by monotonicity of $\VV$,
  \begin{displaymath}
  \VV(D^\epsilon) \leq
  \VV((1-\epsilon)D^y) \leq \E \int_0^\infty
  V((1-\epsilon)\delta^y)\,d\mu
  \end{displaymath}
  so that for~\eqref{eq:50} it suffices to argue that
  $V((1-\epsilon)\delta^y)$ is $\P \otimes \mu$-integrable. To this
  end, we use the asymptotic elasticity condition~\eqref{eq:8}
  which, in conjunction with the monotonicity of $V$, gives
  \begin{align*}
   V((1-\epsilon)\delta^y) &\leq V((1-\epsilon)(\delta^y \wedge
   D^\gamma)) \\&\leq (1-\epsilon)^{\frac{\gamma}{1-\gamma}} V(\delta^y
   \wedge D^\gamma) = (1-\epsilon)^{\frac{\gamma}{1-\gamma}}
   V(\delta^y) \vee V(D^\gamma).
  \end{align*}
  Hence, the integrability claim of~\eqref{eq:50} follows since
  $V(\delta^y) \in \mathbf{L}^1(\P\otimes\mu)$ by choice of $\delta^y$
  and $V(D^\gamma) = U(C^\gamma)-C^\gamma U'(C^\gamma) \in
  \mathbf{L}^1(\P\otimes\mu)$ by assumption on~$C^\gamma$.

  Because of~\eqref{eq:50}, we can apply Theorem~\ref{thm:1} part
  3. to deduce that there is a unique $C^\epsilon \in \cC$ such that
  $\VV(D^\epsilon)=\UU(C^\epsilon)-\E\abr{C^\epsilon,D^\epsilon}$. Moreover,
  \eqref{eq:16} of Theorem~\ref{thm:1} gives $\VV(D^y) \geq
  \UU(C^\epsilon)-\E\abr{C^\epsilon,D^y}$. Recalling the minimality of
  $\VV(D^y)$ we thus obtain
 \begin{displaymath}
   0 \leq \VV(D^\epsilon)-\VV(D^y) \leq
   \E\abr{C^\epsilon,D^y-D^\epsilon}=\epsilon \E\abr{C^\epsilon,D^y-D}.
 \end{displaymath}
 Therefore,
 \begin{equation}
   \label{eq:51}
   0 \leq \E\abr{C^\epsilon,D} \leq \E\abr{C^\epsilon,D^y} \leq \frac{1}{1-\epsilon}\E\abr{C^\epsilon,D^\epsilon},
 \end{equation}
 where the last estimate is immediate from $D^y \leq
 D^\epsilon/(1-\epsilon)$.  Hence, \eqref{eq:49} will follow from
 letting $\epsilon \downarrow 0$ in~\eqref{eq:51} once we have
 established that
  \begin{equation}
    \label{eq:52}
    \E\abr{C^x,D} \leq \liminf_{\epsilon \downarrow 0} \E\abr{C^\epsilon,D}
  \end{equation}
  and
  \begin{equation}
    \label{eq:53}
    \lim_{\epsilon \downarrow 0} \E\abr{C^\epsilon,D^\epsilon} =xy.
  \end{equation}
  To obtain this it suffices to consider a sequence $\epsilon_n
  \downarrow 0$ and prove
  \begin{equation}
    \label{eq:54}
     \dist(C^{\epsilon_n},C^x) \to 0 \text{ as } {n\uparrow\infty} 
  \end{equation}
   for the distance $\dist$ of~\eqref{eq:2} and
  \begin{equation}
    \label{eq:55}
     C^{\epsilon_n} U'(C^{\epsilon_n}) \longto_{n \uparrow \infty} C^x
     U'(C^x)=C^x\delta^y \text{ in $\mathbf{L}^1(\P\otimes\mu)$}.
  \end{equation}
  Indeed, the lower semi-continuity of the bracket $\E\abr{.,D}$ with
  respect to convergence in $\dist$ (Lemma~\ref{lem:11}) then
  yields~\eqref{eq:52}. Similarly~\eqref{eq:53} follows
  because~\eqref{eq:55} yields
  \begin{displaymath}
    \E \abr{C^x,D^y} = \E \int_0^\infty C^x\delta^y\,d\mu =
    \lim_{n} \E\int_0^\infty
    C^{\epsilon_n}U'(C^{\epsilon_n})\,d\mu = \lim_{n} \E\abr{C^{\epsilon_n},D^{\epsilon_n}} 
  \end{displaymath}
  and because~\eqref{eq:47} yields that $\E\abr{C^x,D^y}=xy$ by choice of $x$
  and $y$.

 For~\eqref{eq:54} we will in fact prove that $\delta^{\epsilon_n}
 \set U'(C^{\epsilon_n}) \to \delta^y=U'(C^x)$ in
 $\mathbf{L}^0(\P\otimes\mu)$. 
 If this convergence fails there is $\epsilon>0$ such that
  \begin{displaymath}
    \limsup_n \P\otimes\mu\left[|\delta^{\epsilon_n}-\delta^y|>\epsilon\right]>\epsilon.
  \end{displaymath}
  Observe now that by strict convexity of $V$,  $\delta^n \set
  \frac{1}{2}(\delta^{\epsilon_n}+\delta^y) \in \bdot{\cD}(y)$ satisfies 
  \begin{displaymath}
  V(\delta^n) \leq \frac{1}{2}(V(\delta^{\epsilon_n})+V(\delta^y))  
  \end{displaymath}
  and, for some sufficiently small $\eta>0$, also
  \begin{displaymath}
    \limsup_n \P\otimes\mu\left[V(\delta^n) \leq
      \frac{1}{2}(V(\delta^{\epsilon_n})+V(\delta^y)) -\eta\right]>\eta.
  \end{displaymath}
 Upon integration with respect to $\P\otimes\mu$ we obtain the contradiction
  \begin{align*}
   \mathbf{v}(y) &\leq \limsup_n \E\int_0^\infty V(\delta^n)\,d\mu 
   \\&\leq \limsup_n \frac{1}{2}\E\int_0^\infty
   \left(V(\delta^{\epsilon_n})+V(\delta^y)\right)\,d\mu-\eta^2\\
  &=\limsup_n \frac{1}{2}(\VV(\epsilon_n
  D+(1-\epsilon_n)D^y)+\VV(D^y))-\eta^2\\
& \leq \VV(D^y)-\eta^2<\mathbf{v}(y)
  \end{align*}
  where the last but one estimate is due to the upper-semicontinuity
  of the convex function $[0,1] \ni \epsilon \mapsto \VV(\epsilon
  D+(1-\epsilon)D^y)$ at the boundary point 0. Hence, we must have
  indeed that $\delta^{y_n} \to \delta^y$ in
  $\mathbf{L}^0(\P\otimes\mu)$.

  In light of~\eqref{eq:54}, \eqref{eq:55} will follow once we have
  established the uniform integrability of $(C^{\epsilon_n}
  U'(C^{\epsilon_n}))_{n=1,2,\dots}$. On $\cbr{C^{\epsilon_n} \leq
    C^\gamma}$, we have $C^{\epsilon_n} U'(C^{\epsilon_n}) \leq
  U(C^{\epsilon_n}) \leq U(C^\gamma) \in \mathbf{L}^1(\P \otimes \mu)$
  by assumption on $C^\gamma$. On $\cbr{C^{\epsilon_n} > C^\gamma} =
  \cbr{\delta^{\epsilon_n} <D^\gamma}$, we have $C^{\epsilon_n}
  U'(C^{\epsilon_n}) =-\delta^{\epsilon_n} V'(\delta^{\epsilon_n})
  \leq \frac{\gamma}{1-\gamma} V(\delta^{\epsilon_n})$ by our
  asymptotic elasticity assumption. So it
  suffices to prove the $\mathbf{L}^1(\P \otimes \mu)$-convergence of
  $(V(\delta^{\epsilon_n}))_{n=1,2,\dots}$. Because this sequence is
  convergent in $\mathbf{L}^0(\P \otimes \mu)$ and nonnegative, this
  amounts to showing that $\lim_n \E\int_0^\infty
  V(\delta^{\epsilon_n})\,d\mu = \E\int_0^\infty
  V(\delta^y)\,d\mu$. By Fatou's lemma, we have ``$\geq$" for 
  $\liminf_n$. Recalling that $\E\int_0^\infty
  V(\delta^{\epsilon_n})\,d\mu = \VV(\epsilon_n D+(1-\epsilon_n)D^y)$,
 we deduce ``$\leq$'' for the $\limsup_n$ from the
 upper-semicontinuity of the convex function $[0,1] \ni \epsilon \mapsto \VV(\epsilon
  D+(1-\epsilon)D^y)$ at the boundary point 0. 
\end{proof}

 We now can finally give the 

 \noindent\textbf{Proof of Theorem~\ref{thm:2}.}
 For item 1. we note that $\mathbf{u}$ and $\mathbf{v}$ are
 real-valued by Lemmas~\ref{lem:5} and~\ref{lem:6},
 respectively. Their duality is established in Lemma~\ref{lem:5} and
 their differentiability is contained in Lemmas~\ref{lem:6}
 and~\ref{lem:8}, respectively. The Inada conditions~\eqref{eq:23} can
 be collected from Lemmas~\ref{lem:6} and~\ref{lem:8}. The conjugacy
 relations between optimizers for $\mathbf{u}$ and $\mathbf{v}$ follow
 from the duality of $\mathbf{u}$ and $\mathbf{v}$. Strict concavity
 of $\mathbf{u}$ is similarly a consequence of the differentiability
 of $\mathbf{v}$; see Theorem~26.3 in \citet{Rock:70}.

  Item 3. is just a dual formulation of item 2. For $y>0$,
  Lemma~\ref{lem:4} yields $\delta^y \in \bdot{\cD}(y)$ with
  $\mathbf{v}(y)=\E \int_0^\infty V(\delta^y) \,d\mu$. This readily
  implies that $\breve{D}^y \set \int_.^\infty \delta^y \,d\mu$ is contained
  in $\cD(y)$ and attains the infimum
  in~\eqref{eq:20}. Lemma~\ref{lem:9} shows that $C^x \set
  -V'(\delta^y)$ attains $\mathbf{u}(x)=\UU(C^x)$. Let now $\tilde{D}
  \in \cD(y)$ also attain $\mathbf{v}(y)=\VV(\tilde{D})$. We then have
  \begin{align*}
    \UU(C^x)=\mathbf{u}(x)=\mathbf{v}(y)+xy
    \geq\VV(\tilde{D})+\E\abr{C^x,\tilde{D}},
  \end{align*}
  i.e. $\tilde{D}$ attains the infimum~\eqref{eq:15} for $\hat{C} \set
  C^x$. It thus follows by item 2. of Theorem~\ref{thm:1} that
  $\tilde{D}$ has an envelope process whose density coincides with
  $-U'(C^x)=\delta^y$. Hence, the envelope process of all the
  minimizers of~\eqref{eq:20} is the same process $\breve{D}^y$. This
  accomplishes our proof.  \qed




\appendix

\section{Some stochastic envelope processes}
\label{sec:convex-envelopes}

The existence of envelope processes $\breve{D}$ with~\eqref{eq:12}
and~\eqref{eq:13} for $D \in \cD$ is key for our approach. We show
below how to obtain such an envelope from a result
in~\citet{BankElKaroui04}. Uniqueness is established by an optimal
stopping argument which we adopt from~\citet{BankFoellmer03}.

\begin{Lemma}
  \label{lem:10}
  Under Assumption~\ref{asp:1}, any $D \in \cD$ has a unique (up to
  indistinguishability) envelope process $\breve{D}$ of the form
  \begin{equation}
    \label{eq:56}
    \breve{D}_t = \int_t^\infty U'(C^{\breve{D}})\,d\mu, \; t \geq 0, \text{ for some }
    C^{\breve{D}} \in \cC
 \end{equation}
 such that $\P$-a.s.
 \begin{equation}
   \label{eq:57}
   \opt{\breve{D}}_t \leq \opt{D}_t \text{ for any $t \geq 0$, with ``$=$" if } dC^{\breve{D}}_t>0.
 \end{equation}
\end{Lemma}
\begin{proof}
  For existence we will employ Theorem~2 of \cite{BankElKaroui04},
  which, however, we cannot directly apply with $X \set \opt{D}$
  because $\opt{D}$ may not be of class (D) . So let $S^n \set
  \inf\cbr{t \geq 0 \;:\; D_t \leq n}$ and put $X^n \set
  \opt{D_{. \vee S^n}}$ for $n=1,2,\dots$. Then, because $D$ is
  right-continuous and non-increasing, $X^n$ is even bounded
  and clearly lower-semicontinuous in expectation with
  $X^n_\infty=0$. Moreover, let
\begin{equation*}
  f_t(\omega,l) \set 
  \begin{cases}
    U'_t(\omega,-1/l), & l<0,\\
   -l, & l \geq 0.
  \end{cases}
\end{equation*}
Then, by the properties of $U$:
\begin{itemize}
\item $l \mapsto f_t(\omega ,l)$ is a continuous function, strictly
  decreasing from $+\infty$ to $-\infty$ in $l \in
  (-\infty,\infty)$ for any $(\omega,t) \in \Omega \times [0,\infty)$,
  and
\item $(\omega,t) \mapsto f_t(\omega,l)$ is a predictable $\P \otimes
  \mu$-integrable process on $\Omega \times [0,\infty)$ for any $l \in
  (-\infty,\infty)$.
\end{itemize}
So, by Theorem~2 of \cite{BankElKaroui04} and their Remark~2.1, there exists an optional
process $L^n$ such that
\begin{equation*}
  X^n_S = \condexp[\cF_S]{\int_S^\infty f_t(\sup_{v \in [S,t)} L^n_v) \,d\mu_t}
\end{equation*}
for any stopping time $S \geq 0$. Clearly, we may assume that
$L^n=L^{n+1}$ on $(S^n,\infty)$. So 
\begin{equation*}
L_t \set 
\begin{cases}
L^n_t, & t \in (S^n,\infty), \; n=1,2,\dots,\\
-\infty, & t \in [0,S^\infty],
\end{cases}
\end{equation*}
where $S^\infty = \inf_n S^n =\inf\cbr{t \geq 0 \;:\;
  \opt{D}_t<\infty}$,  consistently defines an
optional process $L$ such that
\begin{equation*}
  \opt{D}_S = \condexp[\cF_S]{\int_S^\infty f_t(\sup_{v \in [S,t)}
    L_v) \,d\mu_t}
\end{equation*}
for any stopping time $S \geq 0$. 

Let us next argue that $L \leq 0$ up to
indistinguishability. Otherwise there exists, by Meyer's optional
section theorem, a stopping time $S$ such that $L_S>0$ on
$\cbr{S<\infty}$ where the latter set has positive probability. But
then we obtain, by definition of $f$,
\begin{equation*}
  0 \leq \opt{D}_S = \condexp[\cF_S]{\int_S^\infty f_t(\sup_{v \in [S,t)}
    L_v) \,d\mu_t}  \leq -L_S \condexp[\cF_S]{\mu([S,\infty))}<0
\end{equation*}
on $\cbr{S<\infty}$, a contradiction.

It follows that 
\begin{equation*}
  C^{\breve{D}}_t \set
 \begin{cases}
  0, & t \in [0,S^\infty],\\
 -1/\sup_{s \in [0,t)} L_s, & t \in (S^\infty,\infty],
 \end{cases}
\end{equation*}
and $\breve{D} \set \int_.^\infty U'(C^{\breve{D}}) \,d\mu$ yield
processes contained in $\cC$ and $\cD$, respectively, with the desired
properties~\eqref{eq:56} and~\eqref{eq:57}.

Let us now prove uniqueness of such a $\breve{D}$ and take an
arbitrary $\tilde{C} \in \cC$ such that $\tilde{D}=\int_.^\infty
U'(\tilde{C})\,d\mu \in \cD$ satisfies $\opt{\tilde{D}} \leq \opt{D}$,
with ``$=$'' on $\cbr{d\tilde{C}>0}$. We will show that, for any
$l>0$, $\tilde{S}^l \set \inf\cbr{t \geq 0 \;:\; \tilde{C}_t>l}$ is
the largest stopping time minimizing $\E\left[D_S - \int_S^\infty
  U'(l) \,d\mu\right]$ over all stopping times $S \geq 0$. As a
result, the level passage times for $\tilde{C}$ are uniquely
determined and, thus, have to coincide with those of $C^{\breve{D}}$,
proving that $\tilde{C}=C^{\breve{D}}$, i.e., $\tilde{D}=\breve{D}$ up
to indistinguishability.

For our optimal stopping claim, we first note that $0 \leq U'(l) \leq
U(l)/l \in \mathbf{L}^1(\P\otimes\mu)$ for $l>0$ and so the above
optimal stopping problem is well-defined. Now take a stopping time $S
\geq 0$ and observe that
\begin{align*}
  \E\left[D_S - \int_S^\infty U'(l) \,d\mu\right] 
& \geq  \E\left[\int_S^\infty\cbr{U'(\tilde{C})- U'(l)} \,d\mu\right] \\
& \geq  \E\left[\int_{\tilde{S}^l}^\infty\cbr{U'(\tilde{C})- U'(l)} \,d\mu\right] 
\end{align*}
where the first inequality is due to $\opt{D} \geq \opt{\tilde{D}}$
and the second follows by definition of $\tilde{S}^l$ and monotonicity of
$c\mapsto U'(c)$. For $S=\tilde{S}^l$ the properties of $\tilde{C}$ ensure
that we have equality everywhere in the above estimates and so
$\tilde{S}^l$ solves our optimal stopping problem. Moreover, the
strict monotonicity of $c \mapsto U'(c)$ ensures that any stopping
time $S>\tilde{S}^l$ will yield a strict inequality in the last estimate
above and so $\tilde{S}^l$ is in fact the largest solution to the
stopping problem, as remained to be shown.
\end{proof}

\section{Convex compactness and a minimax theorem}
\label{sec:minimax}

In this section we first collect a few properties of subsets of $\cC$
related to the pairing~\eqref{eq:1}. In particular, we investigate the
induced notion of convex compactness. For the sake of completeness, we
also provide a version of the well-known minimax theorem which is
adapted to this generalized notion of compactness.

\begin{Lemma}\label{lem:11}
  The pairing $(C,D) \mapsto \E \abr{C,D}$ is lower-semicontinuous
  with respect to convergence in the metric $\dist$ of~\eqref{eq:2} in
  each of its factors.
\end{Lemma}
\begin{proof}
The argument for lower-semicontinuity with respect to $D$ being
similar, let us show lower-semicontinuity with respect to $C
\in \cC$ for fixed $D \in \cD$.

By Fatou's lemma we have
\begin{displaymath}
  \liminf_n \E\abr{C^n,D} = \liminf_n \E \int_{(0,\infty]} C^n \,|dD|
  \geq \E \int_{(0,\infty]}\liminf_n C^n \,|dD|.
\end{displaymath}
Now $\dist(C^n,C)\to0$ implies $\lim_n C^n = C$ on $\cbr{\Delta C=0}$
whose countable complement is a $|dD|$-null set $\P$-almost surely if
$D$ is continuous.

An arbitrary $D \in \cD$ is right-continuous and non-increasing.  We
thus can find continuous, real-valued $D^m \in \cD$ with $D^m \nearrow
D$ pointwise as $m \uparrow \infty$. So, since our claim holds for these
continuous $D^m$, we can conclude
\begin{displaymath}
  \liminf_n \E\abr{C^n,D} \geq \liminf_n \E\abr{C^n,D^m} \geq \E
  \abr{C,D^m} = \E \int_{[0,\infty)} D^m \,dC
\end{displaymath}
for $m=1,2,\dots$.  The claim for $D$ then follows by monotone
integration as we let $m \uparrow \infty$ in the last term of the
above inequality.
\end{proof}

Recall from~\citet{Zitkovic:10}, Definition~2.1, that a subset of a
topological vector space is convexly compact if it satisfies the
finite intersection property for closed and convex subsets.
Equivalently, a closed and convex subset of a topological vector space
is convexly compact if and only if for every net in this set there exists a
convergent subnet of convex combinations (cf. Proposition~2.4 in
\cite{Zitkovic:10}).

We use convex compact sets in the Minimax Theorem~\ref{thm:3} below.
The connection with our duality framework of Lemma~\ref{lem:5} is made
possible by the following result.

\begin{Lemma}\label{lem:12} 
  Let $\cA$ be a convex subset of the consumption space $\cC$ that is
  closed in the topology generated by the metric $\dist$
  of~\eqref{eq:2}.  Then $\cA$ is convexly compact if and only if the
  set of random variables $\cbr{C_\infty \;:\: C \in \cE}$ is bounded
  in probability.

  In particular, for any $c \in [0,\infty)$, $\cbr{C \in \cC \;:\:
    C_\infty \leq c}$ is a convexly compact subset of the space of
  left-continuous processes with bounded total variation endowed with
  the metric $\dist$.
\end{Lemma} 
\begin{proof}
  The proof combines well-known techniques from \citet{Zitkovic:10}
  and \citet{DelbSch:94}.  The details of how to modify
  these techniques to our space of controls $\cC$ can be found in
  Theorem~3.3. in \citet{Kauppila:10}.

  The first step is to show that sets bounded in probability are
  convexly compact.  Lemma~A1.1 in \cite{DelbSch:94}
  illustrates how a (generic) strictly concave functional on the space
  of interest (in our case the space of consumption plans) can be used
  to establish convergence of a subsequence of convex combinations.
  With minor modifications the technique can be used for nets as well.

  The second part is to show that convexly compact sets are bounded in
  probability.  Theorem~3.1 in \citet{Zitkovic:10} proves that closed
  and convex subsets of $L_+^0$ are convexly compact if and only if
  the set is bounded in probability.  The ``only if''-part of this
  theorem can be adapted to show that convexly compact subsets of the
  consumption space are bounded in probability.
\end{proof}

We finish by noting a version of the common minimax theorem which uses
convex compactness and follows with appropriate modifications from the
basic outline of Theorem~3.1 in~\citet{Simons:98}:

\begin{Theorem}\label{thm:3} Let $\cA$ be a nonempty convex,
closed and convexly compact subset of a topological vector space and
let $\cB$ be a nonempty convex subset of another topological
vector space. Let furthermore
  \begin{align*}
    \HH:\cA \times \cB & \to (-\infty,\infty)\\
    (A,B)&\mapsto \HH(A,B)
  \end{align*}
 be concave and upper-semicontinuous  in $A \in \cA$ for $B \in \cB$ fixed,
 and convex in $B \in \cB$ for $A \in \cA$ fixed. 

 Then we have the minimax relation
 \begin{equation}\label{eq:58}
   \sup_{A \in \cA} \inf_{B \in \cB} \HH(A,B) =  \inf_{B \in \cB}
   \sup_{A \in \cA} \HH(A,B). 
 \end{equation}
\end{Theorem}
\begin{proof}
  It is easy to see that ``$\leq$'' holds true in~\eqref{eq:58}. For
  the proof of ``$\geq$'' we let $\alpha \set \inf_{B \in \cB} \sup_{A
    \in \cA} \HH(A,B)$ and we will show that
   \begin{equation*}
     \cbr{A \in \cA \;:\; \HH(A,B) \geq \alpha}, \quad B  \in \cB,
  \end{equation*}
  is a collection of closed convex subsets of $\cA$ which satisfies
  the finite intersection property. Convex compactness of $\cA$ then
  implies that
  \begin{equation*}
    \bigcap_{B \in \cB} \cbr{A \in \cA \;:\; \HH(A,B) \geq \alpha}
    \not= \emptyset,
  \end{equation*}
  i.e., there is $A^* \in \cA$ such that $\inf_{B \in \cB} \HH(A^*,B)
  \geq \alpha$ and, thus, ``$\geq$'' must hold in~\eqref{eq:58} as
  claimed.

  By upper-semicontinuity and concavity of $\HH$ in its first
  variable, each of the level sets $\cbr{A \in \cA \;:\; \HH(A,B) \geq
    \alpha}$, $B \in \cB$, is closed and convex. To prove the finite
  intersection property consider $B_1,\dots,B_m \in \cB$ and observe
  that by the Mazur-Orlicz Theorem (Lemma~2.1~(b) in~\cite{Simons:98})
  there are weights $\lambda_1,\dots,\lambda_m\geq 0$ with
  $\sum_{i=1}^m \lambda_i=1$ such that
  \begin{equation*}
    \sup_{A \in \cA} \cbr{\HH(A,B_1) \wedge \dots \wedge \HH(A,B_m)} =  
    \sup_{A \in \cA} \cbr{\lambda_1\HH(A,B_1) + \dots +\lambda_m \HH(A,B_m)}. 
  \end{equation*}
  By assumption $\HH(A,.)$ is convex for any $A \in \cA$ and so the
  preceding identity entails
  \begin{equation*}
    \sup_{A \in \cA}\cbr{\HH(A,B_1) \wedge \dots \wedge \HH(A,B_m)} \geq  
    \sup_{A \in \cA} \HH(A,\lambda_1 B_1+ \dots +\lambda_m B_m)] \geq \alpha. 
  \end{equation*}
  The finite intersection property thus follows once we have shown
  that the first supremum is actually attained. So let $\HH^{\wedge}(A) \set
  \HH(A,B_1) \wedge \dots \wedge \HH(A,B_m)$, $A \in \cA$, and
  consider a maximizing sequence $A_1, A_2, \dots \in \cA$ for $
  \sup_{A \in \cA}\HH^{\wedge}(A)$. Because $\cA$ is convexly compact there
  is a convergent subnet of finite convex combinations, i.e., there is
  a convergent net $(A_e)_{e \in E}$ of $A_e = \sum_n \gamma^e_n A_n$
  with $\gamma^e_n=0$ for $n \geq N_e$ and $\sum_n \gamma^e_n =1$ such
  that, in addition, for any $N=1,2,\dots$ there is an $e_N \in E$
  with $\gamma^e_n = 0$, $n=0,\dots,N$, for any $e \succeq e_N$; see
  \citet{Zitkovic:10}, Definition~2.3. By concavity of $\HH$ with
  respect to its first variable, also $\HH^{\wedge}$ is concave and so
  \begin{equation*}
    \HH^{\wedge}(A_e)     \geq \sum_n \gamma^e_n     \HH^{\wedge}(A_n) .
  \end{equation*}
  The upper-semicontinuity of $\HH$ in its first variable entails the
  upper-semi\-continuity of $\HH^{\wedge}$. This allows us to conclude in
  the limit that $A_0 \set \lim_{e \in E} A_e \in \cA$ attains $\sup_{A
    \in \cA} \HH^{\wedge}(A)$.
\end{proof}

\bibliographystyle{imsart-nameyear}
\bibliography{../finance}

\end{document}